\documentclass[12pt]{article}
 \expandafter\let\csname equation*\endcsname\relax
 \expandafter\let\csname endequation*\endcsname\relax
\usepackage{bm}
\usepackage{enumerate}
\usepackage[top=1in, bottom=1in, left=1in, right=1in]{geometry}
\usepackage{mathrsfs,color}
\usepackage{amsfonts}\usepackage{multirow}
\usepackage{amssymb}
\usepackage{dsfont}
\usepackage{caption,comment}
\usepackage{blkarray}
\usepackage{amsmath}
\usepackage{float}
\usepackage{amssymb,amsthm}
\usepackage[toc,page]{appendix}
\usepackage{graphicx,afterpage}
\usepackage{epstopdf}
\usepackage{cancel}
\usepackage{mathdots}
\usepackage[pdftex,colorlinks=true]{hyperref}
\usepackage{algorithm}

\usepackage[T1]{fontenc}
\usepackage[utf8]{inputenc}
\usepackage{authblk}

\usepackage{todonotes}
\usepackage[noend]{algpseudocode}
\usepackage{mhequ}
\hypersetup{CJKbookmarks,%
bookmarksnumbered,%
colorlinks,%
linkcolor=black,%
citecolor=black,%
plainpages=false,%
pdfstartview=FitH}
\numberwithin{equation}{section}
\numberwithin{figure}{section}

\newcommand\tabcaption{\def\@captype{table}\caption}

\newtheorem{thm}{Theorem}[section]
\newtheorem{cor}[thm]{Corollary}
\newtheorem{lem}[thm]{Lemma}

\newtheorem{aspt}[thm]{Assumption}
\newtheorem{rem}[thm]{Remark}

\newcommand{\Riccati}{\mathcal{R}}

\newcommand*\Let[2]{\State #1 $\gets$ #2}

\newcommand{\Xhat}{\widehat{X}}

\newcommand{\Chat}{\widehat{C}}

\newcommand{\E}{\mathbb{E}}
\newcommand{\Prob}{\mathbb{P}}

\newcommand{\rmi}{\mathbf{i}}

\newcommand{\reals}{\mathbb{R}}

\newcommand{\unit}{\mathds{1}}

\algrenewcommand\algorithmicrequire{\textbf{Initialization:}}

\newcommand{\Kalman}{\mathcal{K}}

\newcommand{\Xbar}{\overline{X}}

\newcommand{\bfD}{\mathbf{D}}
\newcommand{\Rhat}{\widehat{R}}

\newcommand{\Khat}{\widehat{K}}
\newcommand{\Rtilde}{\widetilde{R}}

\newcommand{\Shat}{\widehat{S}}

\newcommand{\bfP}{\mathbf{P}}
\newcommand{\bfPhat}{\widehat{\mathbf{P}}}

\newcommand{\calU}{\mathcal{U}}

\newcommand{\Gain}{\mathcal{G}}

\begin{document}
\title{Performance of Ensemble Kalman filters in large dimensions}
\author[1]{Andrew J.  Majda \thanks{jonjon@cims.nyu.edu}}
\author[2]{Xin T. Tong \thanks{mattxin@nus.edu.sg}}
\affil[1]{Courant Institute of Mathematical Sciences, New York University}
\affil[2]{Department of Mathematics, National University of Singapore}
\date{\today}
\maketitle
\begin{abstract}
Contemporary data assimilation often involves more than a million prediction variables. Ensemble Kalman filters (EnKF) have been developed by geoscientists. They are successful indispensable tools in science and engineering, because they allow for computationally cheap low ensemble state approximation for extremely large dimensional turbulent dynamical systems. The practical finite ensemble filter like EnKF necessarily involve modifications such as covariance inflation and localization, and it is a genuine mystery why they perform so well with small ensemble sizes in large dimensions. This paper provides the first rigorous stochastic analysis of the accuracy and covariance fidelity of EnKF in the practical regime where the ensemble size is much smaller than the large ambient dimension for EnKFs with random coefficients. A challenging issue overcome here is that  EnKF in huge dimensions introduces unavoidable  bias and model errors which need to be controlled and estimated.
\end{abstract}

\section{Introduction}
With the growing importance of accurate predictions and the expanding availability of data in geoscience and engineering, data assimilation for high dimensional dynamical systems has never been more crucial. The ensemble Kalman filters (EnKF) \cite{evensen03, bishop01,And01, TAB03} are ensemble based algorithms well designed for this purpose. They quantify the uncertainty of an underlying system using the sample information of an ensemble $\{X^{(k)}_n\}^K_{k=1}$. In most applications, the ensemble size $K\sim 10^2$ is much smaller than the ambient space dimension $d\sim 10^6$, thereby significantly reduces the computational cost. The simplicity of EnKF and its close to accurate performance has fueled  wide applications in various fields of geophysical science \cite{MH12,kalnay03}.
\par
However, why does EnKF work well with a small ensemble has remained a complete mystery.  The existing theoretical EnKF performance literature focuses mostly on the impractical scenario where the ensemble size goes to infinity \cite{mandel2011convergence, LTT14}. When the ensemble size is finite,  the only known recent results are wellposedness, nonlinear stability, and geometric ergodicity \cite{KLS14,TMK15non, KMT15, TMK15}. For continuous EnKF analogues, similar results and filter error bounds are obtained when the forecast model admits uniform contraction \cite{DT16, DKT16a, DKT16b}. These theories are inadequate to explain EnKF performance in practice. On the other hand, the practitioners often attribute the EnKF success to the existence of a low effective filtering dimension. This means most of the filtering uncertainty is contained in $p\sim 10$ directions, so the ensemble size is still comparatively large. But the definition of the effective filtering dimension has remained elusive, as its associated subspace often evolves with the dynamics. How  the filter ensemble is attracted to this subspace is even more puzzling. 
\par
Another layer of complexity that shrouds the study of EnKF is its practical variants. The sampling formulation of EnKF leaves it many structural flaws. Meteorologists and engineers have invented various methods to remedy these problems, including square root formulations, covariance inflation, and localization \cite{And07, And08, LDN08, LKMD09}. The derivations of these methods rely purely on physical or statistical intuition, where the filter ensemble is often assumed to be Gaussian distributed. How these methods contribute to the bias and model error introduced by EnKF formulation in practical scenarios lacks rigorous explanation. 
\par
As a consequence of these two theoretical gaps, there are no existing theoretical guidelines on how to choose the ensemble size,  or which and how augmentation should be implemented. This paper intends to fill these gaps by developing a rigorous error analysis framework for a properly augmented EnKF. We study these issues here in the challenging context of Kalman filter with random coefficients \cite{Bou93} in large dimensions \cite{MT16uq}. Although this setting is less difficult than the fully nonlinear case  \cite{TMK15non, TMK15}, it is much richer than the deterministic case with applications to stochastic turbulence models \cite{MH12,MT16uq}.

The remainder of this introduction outlines the objective, strategy, and main results that are proved in detail in the remainder of the paper. We note that \cite{MT16uq} provides a simple rigorous treatment of the same issues without the effect of finite ensemble. 
\subsection{Intrinsic performance criteria}
\label{sec:introcriteria}
 Consider the following signal-observation system with random coefficients
\begin{equation}
\label{sys:random}
\begin{gathered}
X_{n+1}=A_n X_n+B_n+\xi_{n+1},\quad \xi_{n+1}\sim \mathcal{N}(0, \Sigma_n);\\
Y_{n+1}=H_n X_{n+1}+\zeta_{n+1},\quad \zeta_{n+1}\sim \mathcal{N}(0,I_q).
\end{gathered}
\end{equation}
 We assume the signal variable $X_n$ is of dimension $d$, the observation variable $Y_n$ is of dimension $q\leq d$.   The realizations of the dynamical coefficients $(A_n, B_n,\Sigma_n)$, the observation matrix $H_n$, as long as $Y_n$ are assumed to be available, and the objective is to estimate $X_n$. The random coefficients setting allows us to model  internal dynamical intermittencies, where the effective low dimensional subspace may evolve non-trivially.
\par
The optimal filter for system \eqref{sys:random} is the Kalman filter \cite{Bou93, LS01}, assuming $(X_0,Y_0)$ is Gaussian distributed. It estimates $X_{n}$ with a Gaussian distribution $\mathcal{N}(m_n, R_n)$, where the mean and covariance follow a well known recursion:
\begin{equation}
\label{sys:optimal}
\begin{gathered}
\text{Forecast:}\,\,\hat{m}_{n+1}=A_n m_n+B_n,\quad\Rhat_{n+1}=A_n R_nA_n^T+\Sigma_n\\
\text{Assimilation:}\,\, m_{n+1}=\hat{m}_{n+1}+\Gain_n(\Rhat_{n+1})(Y_{n+1}-H_n \hat{m}_{n+1}),\quad R_{n+1}=\Kalman_{n}(\Rhat_{n+1}),\\
\Gain_{n}(C)=C H_n(I_q +H_n C H^T_n )^{-1},\quad \Kalman_{n}(C)=C-\Gain_{n}(C)H_nC.\\
\end{gathered}
\end{equation}
Unfortunately, the Kalman filter \eqref{sys:optimal} is not applicable for high dimensional problems, as it has a computational complexity of $O(d^2q)$. Nevertheless, its optimality indicates that $R_n$ directly describes how well system \eqref{sys:random} can be filtered. Moreover, the classical Kalman filter analysis \cite{Bou93} indicates that $R_n$ converges to a stationary solution $\Rtilde_n$ of \eqref{sys:optimal}, which is independent of  filter initializations. In other words, conditions on the stationary solution $\Rtilde_n$ can be viewed as intrinsic filtering criteria for  system \eqref{sys:random}. Moreover, as $\Rtilde_n$ can be estimated or computed in many scenarios, these criteria can be used as filtering guidelines. There is abundant literature on  various notions of stability of Kalman filters and Riccati equations \cite{Buc67, BR72, Buc72, Kal72, Buc75}. 
\par
In particular for our interest, an augmented Kalman filter stationary solution $\Rtilde_n$ can be used to define the effective dimension. As mentioned earlier, it is conjectured that EnKF performs well with small ensemble size, because there are only $p<K$ dimensions of significant uncertainty during the filtering. Throughout this paper, a threshold $\rho>0$ is assumed to separate the significant dimensions from the others. So one way to formulate the low effective dimension requirement would be 
\[
\Rtilde_n \text{ has at most $p$ eigenvalues above } \rho.
 \] 
 The details of the Kalman filter augmentation, along with the additional low dimension requirements of the unstable directions of $A_n$ and $\Sigma_n$ will be given in the formal Assumption \ref{aspt:lowdim}.

\subsection{EnKF with small ensemble sizes}
\label{sec:introEnKF}
In essence, EnKF is a Monte-Carlo simulation of the optimal filter \eqref{sys:optimal}. It utilizes an ensemble $\{X^{(k)}_n\}$ to describe the filtering uncertainty, and uses the ensemble mean $\Xbar_n$ as the estimator for $X_n$. In the forecast step, an forecast ensemble $\{\Xhat^{(k)}_{n+1}\}$ is generated from the posterior ensemble $\{X^{(k)}_n\}$ simulating the forecast model in \eqref{sys:random}. In the assimilation step, the Kalman update rule of \eqref{sys:optimal} applies, while the mean and covariance are  replaced by ensemble versions. This effectively bring down the computation cost. But the small sample formulation brings up four structural issues. Fortunately, there are existing practical variants of EnKF that can alleviate these problems. By incorporating these augmentations, our theoretical results in return provide rigorous justifications for their improvement in practice.
\par
 The first problem is the covariance rank deficiency. The sample covariance of an ensemble of size $K$ can only be of rank at most $K-1$, so  uncertainties on the $d-K+1$ complementary directions are ignored. The observation operator may accumulate all these underestimated uncertainties, and create a huge bias in the assimilation step.  In practice, this problem is usually remedied by adding a constant inflation to the forecast covariance. In our effective dimension formulation, the directions of significant uncertainty in principle should be captured by the ensemble, while the other directions are under represented. Then with an ensemble covariance matrix $C$, it is intuitive to add $\rho I_d$ to it in order to capture the uncertainty in the ambient space, making the effective covariance estimator $C^\rho:=C+\rho I_d$.
 \par
 The second problem comes from the instability of the dynamics. In many realistic models, intermittent instability comes from the unstable directions of the forecast operator $A_n$ or genuine large surges in $\Sigma_n$. This may not be captured by the ensemble at previous step, and it is necessary to include it in the forecast step. Evidently, in order for the effective filtering dimension to be $p$, the dimension of such instability must be below $p$ as well. 
 \par
 The third problem is the covariance decay from sampling effect. The standard EnKF forecast step intends to recover the covariance forecast of \eqref{sys:optimal}, so $\E \Chat_{n+1}=\Rhat_{n+1}$, assuming $C_n=R_n$. However, the Kalman covariance update operator $\Kalman_{n}$ is a concave operator, see Lemma \ref{lem:Kalconcave}. The Jensen's inequality indicates that the average posterior covariance $\E\Kalman_{n}(\Chat_{n+1})$ is below the target covariance $\Kalman_{n}(\Rhat_{n+1})$, see Lemma \ref{lem:corcoarse} for details. The practical remedy is a multiplicative inflation, so the average forecast covariance is $\E\Chat_{n+1}=r\Rhat_{n+1}$ with a $r>1$. The effect of such inflation has been rigorously studies in \cite{FB07} but only for $d=1$. 
\par 
Lastly, small samples may create spurious correlation in high dimensions. As an easy example, let $X^{(k)}$ be $K$ i.i.d. $\mathcal{N}(0, I_d)$ random vectors. Then by the Bai-Yin's law \cite{Ver11}, the spectral norm of the sample covariance is roughly $1+\sqrt{d/K}$, instead of the true value $1$. In our small sample $K\ll d$ setting, the sampling error is simply horrible. In order to remove such correlation, practical remedies often involve localization procedures. Here we simplify this operation by projecting the sample covariance matrix to its $p$ leading eigenvectors. Mathematically speaking, this ensures the forecast sample covariance can concentrate on the right value. Ideally, if the effective dimension is lower than $p$, such spectral projection will not change much of the covariance matrix.  
\par
Section \ref{sec:EnKF} will give the explicit formulation of the prescribed EnKF, where Algorithms \ref{alg:EnKF} summarizes the procedures.

 \subsection{Covariance fidelity and Mahalanobis error}
\label{sec:mahaintro}
Practical filter performance analysis requires concrete estimates of the filter error $e_n=\Xbar_n-X_n$. This appears to be very difficult for EnKF, because the covariance relations in \eqref{sys:optimal} no longer hold deterministically. To make matters worse, the random sampling error  propagates through a nonlinear update rules as in  \eqref{sys:optimal}. So it is nearly impossible to trace the exact distribution of $e_n$, while it is clearly not Gaussian. 
\par
The existing analysis of EnKF focus mostly on its difference with the optimal filter \eqref{sys:optimal} \cite{FB07, mandel2011convergence, LMT11, LTT14, DKT16a, DKT16b, DT16}. These results often require either the ensemble to be Gaussian distributed at each step, or only consider the case where the ensemble size $K$ reaches infinity. Unfortunately, such assumptions are unlikely to hold in  practice, since  there is computational advantage for EnKF only when $K<d$. 
\par
A more pragmatic strategy would be looking for  intrinsic error statistical relations. In particular, it is important to check whether the reduced covariance estimators dominate the real error covariance, as  underestimating error covariance often causes  severe filter divergence. The Mahalanobis norm is a natural choice for this purpose. Given a  $d\times d$ positive definite (PD) matrix $C$, it generates  a Mahalanobis norm on $\reals^d$:
\begin{equation}
\label{eqn:maha}
\|v\|_C^2:=v^T[C]^{-1}v. 
\end{equation}
This norm is central in many Bayesian inverse problems. For example, given the prior distribution of $X$ as $\mathcal{N}(b, C)$, and a linear observation $Y=HX+\xi$ with Gaussian noise $\xi\sim \mathcal{N}(0,\Sigma)$, the optimal estimate is the minimizer of $\|x-b\|^2_C+\|Y-Hx\|^2_\Sigma. $
In our context, it is natural to look at the non-dimensionalized Mahalanobis error $\frac{1}{d}\|e_n\|^2_{C_n^\rho}$. According to the EnKF formulation described above, the true state is estimated by $\mathcal{N}(\Xbar_n, C_n^\rho)$. If the hypothesis holds, $\frac{1}{d}\|e_n\|^2_{C_n^\rho}$ should roughly be of constant value.  And by showing the Mahalanobis error is bounded, we also show the error covariance estimate $C_n^\rho$ more or less captures the real error covariance, which is known as the covariance fidelity.  
\par
The Mahalanobis error also has surprisingly good dynamical properties in large dimensions \cite{MT16uq}. In short, $\|e_n\|^2_{C_n^{\rho}}$ is a dissipative (also called exponentially stable) sequence. This is actually carried by an intrinsic inequality induced by the Kalman update mechanism. In the optimal filter, it is formulated as 
\begin{equation}
\label{eqn:optdiss}
A_n^T(I-\Gain_n(\Rhat_{n+1})H_n)^TR_{n+1}^{-1}(I-\Gain_n(\Rhat_{n+1})H_n)A_n\preceq A_n^T\Rhat^{-1}_{n+1}A_n\preceq R_n^{-1}.
\end{equation}
This was exploited by previous works in the literature \cite{Bou93, RGYU99} to show robustness of Kalman filters and extended Kalman filters in a fixed finite dimension. See for example, the first displayed inequality in the proof of Theorem 2.6 of \cite{Bou93}. One of the major results of this paper can be informally formulated as follows  
\begin{thm}
\label{thm:intro}
When applying the EnKF described in Section \ref{sec:introEnKF} to system \eqref{sys:random} with an effective filtering dimension $p$ and proper inflation parameters $r>1, \rho>0$, there is a constant $\bfD$ such that if $K>\bfD p$, the nondimenionalized Mahalanobis filter error $\E\frac{1}{d}\|e_n\|_{C^\rho_n}$  converges to a constant independent of $d$ and the filter initialization. 
\end{thm}
The detailed conditions and statements are given by Section \ref{sec:results} and Theorem \ref{thm:weak}. This result rigorously explains  the effectiveness of  EnKF in practice. It is also important to note that the condition requires the ensemble size to grow linearly with the effective dimension, instead of exponentially as in the case of particle filters \cite{Del96, CD02}. This makes the EnKF an promising tool for system with medium size $d$, even if there is no low effective dimensionality assumption.

There are  two useful corollaries from Theorem \ref{thm:intro}. First, its proof indicates the EnKF is exponentially stable for initial ensemble shift, see Corollary \ref{cor:expstable} below. In other words, if two ensembles start with the same ensemble spread, but different means, the difference in their mean will decay exponentially fast. Second, when the system noises are of scale epsilon, the EnKF filter error is also of scale epsilon. This is called the filter accuracy of EnKF, and evidently is very useful when the system \eqref{sys:random} is observed frequently. See Corollary \ref{cor:accuracy} below.

\subsection{Framework of the proof: random matrix concentration and Lyapunov functions}
In order to adapt the Mahalanobis error dissipation \eqref{eqn:optdiss} for EnKF, it is essential to recover $A_n^T\Rhat^{-1}_{n+1}A_n\preceq R_n^{-1}$ with $\Rhat_{n+1}$ replaced by the random ensemble forecast matrix $\Chat_{n+1}$. Such problem has rarely been discussed by the random matrix theory (RMT) community, as it involves matrix inversion and non-central ensembles. Fortunately, standard RMT arguments like Gaussian concentration and $\epsilon$-net covering can be applied, and roughly speaking we can show $\Chat^{-1}_{n+1}\preceq \Rhat^{-1}_{n+1}$. Both covariance inflation techniques and the effective dimensionality are the keys to make this possible. The additive inflation keep the matrix inversion nonsingular, and the multiplicative inflation creates enough room for a spectral concentration to occur at the  effective dimensions.

A RMT result, Theorem \ref{thm:RMT}, guarantees that the Mahalanobis error dissipates like \eqref{eqn:optdiss} with high probability. But the rare sampling concurrence may decrease the sample covariance and creates underestimation. Such intermittency  genuinely exists in various stochastic  problems, and the general strategy is to find a Lyapunov function. In this paper, our Lyapunov function exploits the concavity of the Kalman covariance update operator $\Kalman$, and the boundeness of filter covariance from an observability condition.

\subsection{Preliminaries}
The remainder of this paper is organized as follows. Section \ref{sec:EnKF} formulates an EnKF with specified augmentations, which is summarized by Algorithm \ref{alg:EnKF}. Section \ref{sec:results} formulates the low effective dimension Assumption \ref{aspt:lowdim} and uniform observability Assumption \ref{aspt:obcon}. Theorem \ref{thm:weak} shows that these assumptions guarantee the Mahalanobis error of EnKF decays to a constant geometrically fast. Important Corollaries \ref{cor:expstable} and \ref{cor:accuracy} about exponential stability and accuracy follow immediately. Before diving into the proofs, a simple application of our framework to a stochastic turbulence model is given in Section \ref{sec:example}. The main proof components are illustrated in Section \ref{sec:proof}. The noncentral RMT result is located in Section \ref{sec:RMT}.    

Before we start the discussion, here are a few standard notations we will use in the following.  $\|C\|$ denotes the $l_2$ operator norm of a matrix $C$, and $|x|$ is the $l^2$  norm of a vector $x$.  We use $x\otimes x$ to denote the rank $1$ matrix $xx^T$ generated by a column vector $x$. We use $C\in PD (PSD)$ or simply $C$ is PD (PSD) to indicate a symmetric matrix $C$ is positive definite (semidefinite). $[C]_{j,k}$ denotes the $(j,k)$-th coordinate of a matrix $C$, and $[C]_{I^2}$ is the sub-matrix with both indices in a set $I$. And $A\preceq B$ indicates that $B-A\in PSD$. $\lceil a\rceil$ is the smallest integer above a real number $a$. 

Given an ensemble of vectors $x^{(1)},\ldots, x^{(K)}$, we use $\bar{x}=\frac{1}{K}\sum_{k=1}^K x^{(k)}$ to denote its ensemble average. $\Delta x^{(k)}=x^{(k)}-\bar{x}$ to denote the deviation of each ensemble. Some times, it is easier to describe an ensemble in terms of its mean $\bar{x}$, and  spread matrix $S=[\Delta x^{(1)},\ldots, \Delta x^{(K)}]$.

We assume the distribution of  filter initializations is known. Generally speaking, there are no specific requirements for their values. But some results implicitly rely on the invertibility of the covariance matrices. 

Following \cite{Bou93}, we say a random sequence $Z_0,Z_1,\ldots$ is \emph{stationary}, if $(Z_0,Z_1,\ldots)$ and $(Z_k,Z_{k+1},\ldots)$ have the same distribution. We say such sequence is \emph{ergodic}, if there is only one invariant measure for the shifting map $(Z_0,Z_1,\ldots)\mapsto (Z_1,Z_2,\ldots)$.

\section{EnKF with low effective dimensions}
\label{sec:EnKF}
\par
EnKF utilizes an ensemble  $\{X^{(k)}_n\}_{k=1,\ldots, K}$ to describe the underlying uncertainty.   Such formulation effectively brings down the computation cost of each filtering iteration, and have shown good filtering skills through numerous numerical evidences.

\subsection{Forecast Step with instability representation} 
The forecast step propagates the underlying uncertainty of time $n$ to time $n+1$.  Since the effective posterior distribution is $\mathcal{N}(\Xbar_n, C_n^\rho)$, the target forecast distribution should be $\mathcal{N}(A_n\Xbar_n+B_n, A_nC_nA_n^T +\rho A_n A_n^T+\Sigma_n).$ In order to remedy the covariance decay due to sampling, we also multiply the target covariance with a ratio $r>1$. 
\par
 In the ensemble formulation, the posterior ensemble is propagated to the forecast ensemble, $\Xhat^{(k)}_{n+1}=A_nX^{(k)}_{n}+B_n+\xi_{n+1}^{(k)}$. The noise $\xi_{n+1}^{(k)}$ is Gaussian distributed with mean zero, it intends to capture the instantaneous instability of the system. Its covariance $\Sigma_n^+$  will be specified soon. Denote  the inflated forecast ensemble spread $\sqrt{r}A_n S_n+\sqrt{r}[\Delta\xi_{n+1}^{(1)},\ldots, \Delta\xi_{n+1}^{(K)}]$ as $\Shat_{n+1}$, where $S_n$ is the ensemble spread matrix at step $n$:
 \[
 S_n=[\Delta X^{(1)}_n,\cdots, \Delta X^{(K)}_n],\quad \Delta X^{(k)}_n=X^{(k)}_n-\Xbar_n.
 \]
 The corresponding forecast covariance is given by 
\begin{equation}
\label{eqn:enscor}
\Chat_{n+1}=\frac{\Shat_{n+1}\Shat_{n+1}^T}{K-1},\quad \E_n \Chat_{n+1}=rA_n C_n A_n^T+r\Sigma_n^+. 
\end{equation}
Here $\E_n$ denotes expectation conditioned on system and ensemble realization up to time $n$. Since $\Chat_{n+1}$ is an ensemble covariance, it under-represents the ambient space uncertainty. This can be remedied by adding a constant covariance $\rho\tau I_d$. Here $\tau>0$ is a parameter that increases our framework flexibility. Its value depends on the model  setting, see an example in Section \ref{sec:example}. Our framework is valid independent of $\tau$, and in the first reading, $\tau$ can be seen as $1$. In order to preserve covariance fidelity, intuitively we want
 \begin{equation}
\label{eqn:sampledominate}
\E_n \Chat_{n+1}^{\tau\rho}=rA_nC_nA_n^T+r\Sigma^+_n+\rho \tau I_d
\succeq r(A_nC_nA_n^T +\rho A_n A_n^T+\Sigma_n). 
\end{equation}
 For this purpose, let $\bfPhat_n$ be the projection to the eigen-subspace of  $\rho A_n A_n^T+\Sigma_n-\rho \tau/r I_d$ with positive eigenvalues. If we let $\Sigma^+_n$ be the positive part of the prescribed matrix,
 \[
 \Sigma^+_n=\bfPhat_n(\rho A_n A_n^T+\Sigma_n-\rho \tau/r I_d)\bfPhat_n,
 \]
 it is straightforward to check that \eqref{eqn:sampledominate} holds. $\Sigma^+_n$ captures the instantaneous instability of system \eqref{sys:random}. It essentially includes the unstable directions of $A_n$ with singular value above a spectral gap $\sqrt{\tau/r}$, and the directions of $\Sigma_n$ with eigenvalue above $\rho\tau/r$. Its low dimensionality will be another crucial component of the low effective dimension Assumption \ref{aspt:lowdim}.

\subsection{Assimilation step with spectral projection}
Once the forecast covariance is computed, the Kalman gain matrix can be obtained through
\[
G_{n+1}=\Gain_n(\Chat^{\tau\rho}_{n+1})=\Chat^{\tau\rho}_{n+1}H_n^T(I_q+H_{n}\Chat^{\tau\rho}_{n+1} H^T_{n})^{-1}.
\]
We update the mean from $\overline{\Xhat}_{n+1}=\frac1K\sum_k \Xhat^{(k)}_{n+1}$ to 
\begin{align*}
\Xbar_{n+1}&=\overline{\Xhat}_{n+1}+G_{n+1}[Y_{n+1}-H\overline{\Xhat}_{n+1}].
\end{align*}
Based on the classical Kalman formulation, the target posterior covariance should be $\Kalman_{n}(\Chat^{\tau\rho}_{n+1})$, with the Kalman covariance update operator
\begin{align}
\notag
\Kalman_{n}(C)&:=C-C H_n^T(I_q+H_nC H_n^T)^{-1}H_nC\\
\label{eqn:Kalmanoper}
&=(I-\Gain_{n}(C)H_{n})C(I-\Gain_{n}(C)H_n)^T+ \Gain_n(C)\Gain_n(C)^T.
\end{align}

Unfortunately, $\Kalman_{n}(\Chat^{\tau\rho}_{n+1})$ cannot be directly obtained as a posterior ensemble covariance matrix, since it may have rank $d$ instead of $K-1$. Moreover, the forecast sample covariance matrix may have spurious correlation due to lower than dimension sampling size. Due to this reason, we project  $\Kalman_{n}(\Chat^{\tau\rho}_{n+1})$ to its eigenspace associated with the largest $p$ eigenvalues. The posterior ensemble should have an effective covariance $C^\rho_{n+1}$ matching it. If we denote the projection as $\bfP_{n+1}$, let $Q DQ^T $ be the eigenvalue decomposition of $\bfP_{n+1}(\Kalman_{n}(\Chat^{\tau\rho}_{n+1})-\rho I)\bfP_{n+1}$, and $\Psi\Lambda\Phi^T$ be the SVD decomposition of $\Shat_{n+1}$. We can update the posterior spread $S_{n+1}$  through an ensemble adjustment Kalman filter (EAKF) type of update, $S_{n+1}=QD^{1/2}\Lambda^\dagger\Psi^T\Shat_{n+1}$, then
\begin{equation}
\label{eqn:EAKF}
C_{n+1}=\tfrac{S_{n+1}S_{n+1}^T}{K-1}=\bfP_{n+1}(\Kalman_{n}(\Chat^{\tau\rho}_{n+1})-\rho I)\bfP_{n+1}.
\end{equation}
Notice that the directions with eigenvalues of $\Kalman_{n}(\Chat^{\tau\rho}_{n+1})$ below the threshold $\rho$ are filtered out on the right hand side. Moreover,
it is straight forward to verify that
\begin{equation}
\label{eqn:objcor}
\Kalman_{n}(\Chat^{\tau\rho}_{n+1})+\rho I_d\succeq C_{n+1}+\rho I_d\succeq \Kalman_{n}(\Chat^{\tau\rho}_{n+1}). 
\end{equation}
In case one does not have a low effective dimension, so $p=d$, the spectral projection is trivial $\bfP_{n+1}=I_d$. 

Before we move on and analyze EnKF performance, let us briefly discuss the computational complexity involved in the analysis step, assuming $d>K^2>p^2$ and standard numerical procedures \cite{GV83}. It is important to notice that $\Chat_{n+1}=\frac1K\Shat_{n+1}\Shat_{n+1}^T$ is of rank $K$. Moreover the following Woodbury matrix identity holds,
\[
(I_q+H_n\Chat^{\tau\rho}_{n+1}H_n^T)^{-1}=Q_n-Q_n (I_K+\Shat_n^T H_n^T H_n \Shat_n)^{-1} Q_n,
\]
where $I_K+\Shat_n^T H_n^T H_n \Shat_n$ is $K\times K$, and $Q_n=(I_q+\tau\rho H_n H_n^T)^{-1}$  is usually easy to compute, because $H_n$ is often a simple projection. This makes vector products with the matrix 
\[
\Kalman_n(\Chat^{\tau\rho}_{n+1})=\Chat_{n+1}+\tau\rho I_d-(\Chat_{n+1}+\tau\rho I_d) H_n^T (I_q+H_n\Chat^{\tau\rho}_{n+1}H_n^T)^{-1}H_n(\Chat_{n+1}+\tau\rho I_d)
\]
of complexity $O(Kd)$ instead of $O(d^2)$. Consequencely the finding of $p$ largest eigenvalues and associated eigenvectors is of complexity $O(pKd)$ through, for example, the power method. Since the SVD decomposition is identical to standard EAKF, it takes complexity of $O(K^2d)$ \cite{TAB03}. In conclusion, this version of EAKF does not require an additional order of complexity.  We comment that simpler operations may exist with other EAKF formulations. For example \cite{SO08} suggested used letting $S_{n+1}=(I-\frac12G_{n+1} H_n)\Shat_{n+1}$.

\subsection{Summary of the algorithm}
The mathematical formulation of our EnKF can be summarized by Algorithm \ref{alg:EnKF}
\begin{algorithm}
  \caption{EnKF with covariance inflations and spectral projections
    \label{alg:EnKF}}
  \begin{algorithmic}[1]
    \Require{$K$: ensemble size, $p$: effective dimension, $r$: multiplicative inflation, $\rho$ uncertainty significance threshold, $\tau:$ flexibility parameter, $X^{(k)}_0:$ initial ensemble.}
    \For {$n\gets 0\textrm{ to } T-1$}
    \Let {$\Sigma_n^+$}{The positive part of $\rho A_n A_n^T+\Sigma_n-\rho\tau/rI_d$}.
    	\State Generate $\xi^{(k)}_{n+1}\sim \mathcal{N}(0,\Sigma_n^+), k=1,\ldots, K$.
	\Let {$\overline{\Xhat}_{n+1}$}{$A_n \Xbar_n+B_n+\frac{1}{K}\sum_{k=1}^K \xi_{n+1}^{(k)}$,}\quad $\Shat_{n+1}\gets\sqrt{r}(A_n S_n+[\Delta\xi_{n+1}^{(1)},\ldots,\Delta\xi_{n+1}^{(K)}]) $. 
       \Let {$\Chat_{n+1}$}{$\frac{1}{K-1}\Shat_{n+1}\Shat_{n+1}^T$}, \quad $G_{n+1}\gets \Chat_{n+1}^{\tau\rho} H_n^T(I_q+H_{n}\Chat^{\tau\rho}_{n+1} H^T_{n})^{-1}$.
       \Let {$\Xbar_{n+1}$}{$\overline{\Xhat}_{n+1}+G_{n+1}(Y_{n+1}-H_n \overline{\Xhat}_{n+1})$}.
       \Let {$\bfP_{n+1}$}{Projection to the largest $p$ eigenvectors of  $\Kalman_{n}(\Chat^{\tau\rho}_{n+1}).$}
       \Let {$S_{n+1}$}{ $A_nS_n$} by an EAKF type of update, so that 
       \[
	\tfrac{S_{n+1}S_{n+1}^T}{K-1}=\bfP_{n+1}(\Kalman_{n}(\Chat^{\tau\rho}_{n+1})-\rho I_d)\bfP_{n+1}.
       \]
	\State \Return {State estimation: $\mathcal{N}(\Xbar_{n+1}, \tfrac{S_{n+1}S_{n+1}^T }{K-1}+\rho I_d$ )}. 
      \EndFor
  \end{algorithmic}
\end{algorithm}

\section{Main results: EnKF performance}
\label{sec:results}

\subsection{Effective low dimensionality}
As explained in Section \ref{sec:introcriteria},  Kalman filters will be employed to formulate the effective dimensionality. Since our EnKF implements covariance inflation techniques, the associated signal-observation system is also augmented with a stronger inflation \cite{MT16uq}:
\begin{equation}
\label{sys:changed}
\begin{gathered}
X'_{n+1}=rA_nX'_n+B_n+\xi'_{n+1},\\
Y'_{n+1}=H_nX'_{n+1}+ \zeta'_{n+1}.\\
\xi'_{n+1}\sim \mathcal{N}(0, \Sigma'_n=r^2\Sigma_n^+ +r^2\tau \rho I_d),\quad \zeta'_{n+1}\sim \mathcal{N}(0, I_q). 
\end{gathered}
\end{equation}
In principle, $r^2$ can be replaced with any ratio above $r$ and our analysis below still holds. We use $r^2$ just for notational simplicity.  The optimal filter for \eqref{sys:changed} is a Kalman filter. The associated covariance solves follows recursion
\begin{equation}
\label{sys:optimalr}
\begin{gathered}
\Rhat'_{n+1}=r^2A_nR'_nA^T_n+\Sigma'_n, \quad R'_{n+1}=\Kalman_{n}(\Rhat'_{n+1}).
\end{gathered}
\end{equation} 
Given the value of $R'_j$, the value of $R'_k$ at a later time $k\geq j$ can be computed by the recursion above. So there is a $\mathcal{F}^c_k$ measurable mapping $\mathcal{R}_{j,k}$ such that $R'_k=\mathcal{R}_{j,k}(R'_j).$ 

Our low dimensionality reference can be formulated through one solution $\Rtilde_k=\mathcal{R}_{j,k}(\Rtilde_j)$:
\begin{aspt}
\label{aspt:lowdim}
The signal observation system has intrinsic dimension $p$, if for a  (stationary) solution $\Rtilde_n=\Riccati_{k,n}(\Rtilde_k)$ of \eqref{sys:optimalr}:
\begin{itemize}
\item The system instability matrix $\Sigma^+_n$ has rank not exceeding $p \,\,\Leftrightarrow\,\,A_nA_n^T+\Sigma_n/\rho$ has at most $p$ eigenvaleus above $\tau/r$.  
\item $\Rtilde_{n}$ has at most $p$ eigenvalues above $\rho$.  
\end{itemize}
\end{aspt}
Although our framework works for an arbitrary solution of \eqref{sys:optimalr}, in most cases any solution will converge to a unique PD stationary sequence, assuming stationarity, ergodicity, weak observability and controllability of the system coefficients, see \cite{Bou93}. It is also straightforward to verify that, if the original system \eqref{sys:random} satisfies the weak observability and controllability condition of \cite{Bou93}, then so does the augmented system \eqref{sys:changed}. A proof of a similar claim is in the appendix of \cite{MT16uq}. Then it makes more sense to impose Assumption \ref{aspt:lowdim} on this stationary solution, and by doing so, the assumption is independent of the initial conditions. This is why we put \emph{stationary} in bracket, the readers should choose the proper variant in application.  

In order to focus on the more interesting EnKF sampling effect, this paper considers a relative simple system setting, so  \eqref{sys:random} is observable through a fixed time interval $m$:
\begin{aspt}
\label{aspt:obcon}
Suppose there are constants $D_A, D_\Sigma, D_R$ such that 
\[
\|A_n^{-1}\|,\|A_n\|\leq D_A,\quad \|\Sigma_n\|\leq D_\Sigma,\quad
\|\Rtilde_n\|,\|\Rtilde_n^{-1}\|\leq D_R.
\] 
Also suppose there is a constant step size $m$, a constant $c_m$, such that the  observability  Gramian defined below satisfies $\mathcal{O}_m\succeq c_m I_d$,
\[
\mathcal{O}_m=\sum_{k=1}^m A_{k,1}^TH_k^TH_kA_{k,1},\quad A_{k,j}=r^{k-j}A_{k-1}\cdots A_j. 
\]
\end{aspt}
Assumption  \ref{aspt:obcon} will simplify the control of a solution of \eqref{sys:optimalr} a lot, for example Lemma \ref{lem:boucontract} shows that $\Riccati_{k,k+m}(C)$ will always be bounded and uses \cite{Bou93}.
\begin{rem}
It is worth noticing that the direct requirement  on system \eqref{sys:random}, Assumption \ref{aspt:obcon},  is quite weak, given that  $(A_n, B_n, H_n, \Sigma_n)$ can be any random sequence.  However, Assumption \ref{aspt:lowdim} is another implicit condition imposed on system \eqref{sys:random}, since the $\Rtilde_n$ is generated through \eqref{sys:optimalr}. The dependence of $\Rtilde_n$ on $(A_n, B_n, H_n, \Sigma_n)$ is in general very involved, even in the deterministic, time homogenous case. Section \ref{sec:example} will further discuss how to verify Assumption \ref{aspt:lowdim}, and demonstrate with an example. More examples for Kalman filters with random coefficients can be found in \cite{MT16uq} and its references. On the other hand, Assumption \ref{aspt:lowdim} is required in theory, simply because it can be verified offline and before executing the algorithm. In practice, it suffices to check on the fly whether the posterior covariance $\Kalman_n(\Chat^{\tau\rho}_{n+1})$ indeed has less than $p$ eigenvalues above $\rho$. If yes, the projection error is zero, the proof in Section \ref{sec:dissmaha} holds as well with $\chi_{n+1}=1$,  so the same results would hold. 
\end{rem}

%
%
%

\subsection{Covariance fidelity and filter performance}
As discussed in Section \ref{sec:mahaintro}, Mahalanobis error  is a natural statistics for covariance fidelity quantification. By proving a stronger but more technical result, Theorem \ref{thm:strong}, we can show that the Mahalanobis error decays to a constant geometrically fast. With stronger conditions, a similar proof should lead to a similar bound for $\E \|e_n\|^s_{C_n^\rho}$ with $s>1$. But we only demonstrate the case with $s=1$ for exposition simplicity.
\begin{thm}
\label{thm:weak}
Suppose system \eqref{sys:random} satisfies the uniform observability  Assumption \ref{aspt:obcon}, and has an intrinsic filtering dimension $p$ as described in Assumption \ref{aspt:lowdim}. For any $c>0$, there exist a  function $F: PSD\to \reals$, a constant $\bfD$, and a sequence $M_n$,  such that when $K>\bfD p$,
\[
\E \|e_n\|_{C^\rho_n}\leq r^{-\frac{n}{6}}\E F(C_0) \sqrt{\|e_0\|^2_{C_0^\rho}+2m}+M_n\sqrt{d},
\]
\[
\E |e_n|\leq r^{-\frac{n}{6}}\sqrt{D_R+\rho}\E  F(C_0) \sqrt{\|e_0\|^2_{C_0^\rho}+2m}+\sqrt{(D_R+\rho)d}M_n.
\] 
The function $F$ and the sequence satisfy the following bounds with a constant $D_F$:
\[
F(X)\leq D_F\exp(D_F\log^3\|X\|),\quad \limsup_{n\to \infty} M_n\leq \frac{1+c}{1-r^{-\frac{m}{6}}}. 
\]
Moreover, the ensemble covariance $C_n$ is bounded by $\Rtilde_n$ most of the time:
\begin{equation}
\label{eqn:cnbound}
\limsup_{n\to \infty}\E \max\{1, \|C_n\Rtilde_n^{-1}\|^m\}\leq 1+c.
\end{equation}
\end{thm}

\subsection{Exponential stability}
Another useful property implied by the previous analysis is that EnKF is exponentially stable. Let $\{X_0'^{(k)}\}=\{X_0^{(k)}+\Delta\}$ be a shift of the original initial ensemble. If both EnKF generate the same realization for $\xi^{(k)}_n$ and assimilate the same observation sequence $Y_n$, then it is straight forward to check their ensemble spread matrices remain the same, and the difference in their mean estimates is given by
\[
(\Xbar_n-\Xbar_n')= U_{n,0}(\Xbar_0-\Xbar_0'),\quad U_{n,m}=\prod_{k=m}^{n-1} (I-G_{k+1}H_k)A_k. 
\]
So if $\|U_{n,0}\|$ converges to zero exponentially fast, then so does the mean difference. In \cite{Bou93}, this is called exponential stability. 

\begin{cor}
\label{cor:expstable}
Under the conditions of  Theorem \ref{thm:weak}, suppose $\E F(C_0)<\infty$ and $K>\bfD p$,   the EnKF is exponentially stable 
\[
\limsup_{n\to \infty} \frac{1}{n} \log \E\left\|\prod_{k=0}^{n-1} (I-G_{k+1}H_k)A_k\right\|\leq -\frac{1}{6}\log r.
\]
\end{cor}
The proof is located after the proof of Theorem \ref{thm:strong}. 

\subsection{Filter accuracy}
When system \eqref{sys:random} is observed frequently, the system noises are of close to zero scale. Intuitively, the filter error should be close to zero as well. Such property has recently been investigated by \cite{LSS14} for other observers, and called filter accuracy. While filter accuracy may come easily for the Kalman filter \eqref{sys:optimal}, it is highly nontrivial for EnKF. Our framework reveals that filter accuracy of EnKF can be obtained by filter accuracy of the reference Kalman filter.  
\begin{cor}
\label{cor:accuracy}
Under the same conditions of Theorem \ref{thm:weak}, there exists  a  constant  $C$,  such that when $K>\bfD p$, for any $\epsilon>0$ there is an EnKF filter with ensemble size $K$ for the following system 
\begin{equation}
\label{eqn:epsilonnoise}
\begin{gathered}
X_{n+1}=A_n X_n+B_n+\epsilon \xi_{n+1}\\
Y_{n+1}=H_n X_{n+1}+\epsilon \zeta_{n+1},
\end{gathered}
\end{equation}
such that with the same function and sequence in Theorem \ref{thm:weak}:
\[
\E |e_n|\leq \epsilon r^{-\frac{n}{6}} \sqrt{D_R+\rho}\E F(\epsilon^{-2}C_0)\sqrt{\|e_0\|^2_{C_0^\rho}+2m}+\epsilon \sqrt{(D_R+\rho)d} M_n .
\] 
\end{cor}
\begin{proof}
This is a straightforward from Theorem \ref{thm:weak}, since if $\Rtilde_n$ is a reference Kalman filter covariance for system \eqref{sys:random}, then $\epsilon^2\Rtilde_n$ is a reference Kalman filter covariance for system \eqref{eqn:epsilonnoise}, which has spectral norm bounded by $\epsilon^2 D_R$. So if one applies Algorithm \ref{alg:EnKF} with rescaled parameters $(K, p,r, \epsilon^2\rho, \tau)$, one can check both Assumptions \ref{aspt:lowdim} and \ref{aspt:obcon} hold. The $\epsilon^{-2}$ appears before $C_0$, because in fact $F(C_0)=f(\|C_0\Rtilde_0^{-1}\|)$ for some function $f$, which will be clear in Theorem \ref{thm:strong}. 
\end{proof}

\section{A simple application and example}
\label{sec:example}
When applying our results to a concrete problem in the form of \eqref{sys:random}, we need to compute the stationary filter covariance $\Rtilde_n$, as it plays a vital role in the low effective dimension Assumption \ref{aspt:lowdim}, and it controls the ensemble covariance $C^\rho_n$. In practice, this requires nontrivial numerical computation. One the other hand, there is a huge literature on Kalman filters, so there are various ways to simplify the computations. \cite{MT16uq} has discussed a few such strategies,  including unfiltered covariance, benchmark principle, and comparison principle. We will not reiterate these strategies in this paper, but instead apply some of them to a concrete simple turbulence problem \cite{MH12,MT16uq}. 

\subsection{Linearized stochastic turbulence in Fourier domain}
\label{sec:linearflow}
Consider the following stochastic partial differential equation \cite{HM08non,MH12}
\begin{equation}
\label{sys:SPDE}
\partial_t u(x,t)=\Omega(\partial_x) u(x,t)-\gamma(\partial_x) u(x,t)+F(x,t)+dW(x,t).
\end{equation}
For the simplicity of discussion,  the underlying space is assumed to be an one dimensional torus $\mathbb{T}=[0,2\pi]$, while generalization to higher dimensions is quite straight forward. The terms in \eqref{sys:SPDE} have the following physical interpretations: 
\begin{enumerate}[1)]
\item $\Omega$ is an odd polynomial of $\partial_x$. This term usually arises from the Coriolis effect from  earth's rotation, or the advection by another turbulent flow. Suppose $\Omega(\partial_x)e^{\rmi k x}=\rmi \omega_k e^{\rmi k x}$.
\item $\gamma$ is a positive and even polynomial of $\partial_x$. This term models the general diffusion and damping of turbulences. Suppose $\gamma(\partial_x)e^{\rmi kx}=\gamma_k e^{\rmi k x}$.
\item $F(x,t)$ is a deterministic forcing and $W(x,t)$ is a stochastic forcing. 
\end{enumerate}
One way to discretize \eqref{sys:SPDE} in both time and space, is to consider the Galerkin truncation of $u$ at time $nh$ with a fixed interval $h>0$. Let $X_n$ be a $2J+1$ dimensional vector, with its $k$-th component being the $k$-th Fourier component of $u(\,\cdot\, , nh)$, in other words,
\[
u(x,nh)=[X_n]_0+\sum_{1\leq k\leq J}2[X_n]_k\cos(kx) +2[X_n]_{-k}\sin(kx).
\]
Suppose both the deterministic and the stochastic forcing admit a Fourier decomposition:
\[
F(x,t)=2\sum_{1\leq k\leq J} f_k\cos(kx)+f_{-k}\sin(kx),\quad W(x,t)=2\sum_{1\leq k\leq J} \sigma^u_k (W_k(t)\cos(kx)+W_{-k}(t)\sin(kx)) .
\]
Here $W_{\pm k}(t)$ are  standard Wiener processes.
 
A time discretization of \eqref{sys:SPDE} in the Fourier domain yields the system coefficients for the system vector $X_n$. The details are in \cite{MT16uq}, and the results are presented as follows. $A_n=A$ is diagonal with $2\times 2$ sub-blocks, and $\Sigma_n=\Sigma$ is diagonal. Their entries with $B_n$ are:
\begin{equation}
\label{sys:turb}
\begin{gathered}
\left[A\right]_{\{k,-k\}^2}=\exp(-\gamma_k h)\begin{bmatrix} \cos(\omega_k h) &\sin(\omega_k h)\\
-\sin(\omega_k h) & \cos(\omega_k h)
\end{bmatrix},\quad [B]_k=f_k(nh)h,\\
[\Sigma]_{k,k}=\frac{(\sigma_k^u)^2}{2}\int_{nh}^{(n+1)h} \exp(-2\gamma_k s)ds=\frac{1}{2} E_k^u(1-\exp(-2\gamma_k h)).
\end{gathered}
\end{equation}
$E_k^u=\frac{1}{2\gamma_k}(\sigma_k^u)^2$ stands for the  stochastic energy of the $k$-th Fourier mode, and also the sum of stochastic energy of $[X_n]_k$ and $[X_n]_{-k}$.

In practice, the damping often grows and the energy decays like polynomials of the wavenumber $|k|$
\begin{equation}
\label{eqn:physic}
 \gamma_k=\gamma_0+\nu |k|^\alpha,\quad E_k^u=E_0|k|^{-\beta},\quad \alpha> 0,\beta\geq 0.
\end{equation}
To show that our framework is directly computable, we will also consider the following specific set of physical parameters with a Kolmogorov energy spectrum used in \cite{MG13}:
\begin{equation}
\label{eqn:physics}
\alpha=2,\quad \beta=\frac{5}{3},\quad r=1.1,\quad \tau=0.6, \quad h=0.5,\quad\gamma_0=\nu=0.01,\quad E_0=1.
\end{equation}

\subsection{Reference spectral projection}
In order to verify Assumption \ref{aspt:lowdim}, we need to estimate the stationary Kalman covariance $\Rtilde_n$. The unfiltered equilibrium covariance $V'_n$ of $X'_n$ is a crude upper bound for $\Rtilde_n$, as the Kalman filter has the minimum error covariance, while $V'_n$ is the error covariance made by estimating $X'_n$ using its equilibrium mean. Although this is a crude estimate, it works for any choice of observation, and it is easy to compute, since $V'_{n+1}=r^2A_nV'_nA_n^T+\Sigma_n'$. If the system is time homogeneous, $V'_n$ will be a constant $V'$.  $V'$ is an upper bound of $\Rtilde_n$.  \cite{MT16uq} also applies this idea for reduced filtering error analysis. 

In particular for \eqref{sys:turb}, $A_nA_n^T+\Sigma_n$ is a diagonal matrix with entries
\[
[\rho A_nA_n^T+\Sigma_n]_{k,k}=\rho\exp(-2\gamma_k h)+\tfrac{1}{2}E_k^u(1-\exp(-2\gamma_k h)). 
\] 
If $[\rho A_nA_n^T+\Sigma_n]_{k,k}\geq \tau \rho/r$, this mode should be included by the instability covariance $\Sigma_n^+$. For the other Fourier modes, one can verify that $V'$ is a constant diagonal matrix with entries
\[
v'_k=r^2\exp(-2\gamma_k h)v'_k+r^2[\Sigma]_{k,k}+r^2\tau\rho\quad\Rightarrow\quad v'_k=\frac{r^2([\Sigma]_{k,k}+\tau\rho)}{1-r^2\exp(-2\gamma_k h)}.
\]
Since $\Rtilde_n\preceq V'$, by Lidskii's theorem 6.10 of \cite{Kat82}, it suffices to show $V'$ has at most $p$ eigenvalues above $\rho$. In summary, the following must hold
\begin{equation}
\label{eqn:range1}
\rho\geq \max\left\{\frac{r^2[\Sigma]_{k,k}}{1-r^2\tau-r^2\exp(-2\gamma_k h)},\frac{r\rho}{\tau}\exp(-2\gamma_k h)+\frac{r}{2}[\Sigma]_{k,k}\right\}
\end{equation}
except for at most $p$ different $k$. The quantity is easily computable, for example with the physical parameters \eqref{eqn:physics}, $p$ can be set as $15$, or $30$ if the negative wavenumbers are also considered, and $\rho=0.04$. Such small $p$ makes $K\sim 100$ comparatively large, and it is independent of the Galerkin truncation range $J$.

\subsection{Regularly spaced observations}
Observations can significantly decrease the Kalman filter covariance $\Rtilde_n$, so they help keeping the intrinsic filtering dimension low. Here we consider only a simple but useful scenario where the observations of $u(x,t)$ are made  at a regularly spaced network, $x_k=\frac{2\pi k}{2J+1}, k=0,\ldots,2J$, and have an Gaussian observation error $\mathcal{N}(0,\sigma^o)$ at each location. \cite{MH12, MT16uq} have shown that this is equivalent to a direct observation, with $H_n=I_{2J+1}\sqrt{\frac{2J+1}{\sigma^o}}$. Then the reference covariance matrix $\Rtilde_n$ is a constant diagonal  matrix with entries $[\Rtilde]_{k,k}=r_k$. When $k$ is not a mode of instability, it solves a Riccati equation
\[
r_k=\frac{\sigma^o\hat{r}_k}{\sigma^o+(2J+1)\hat{r}_k},\quad\hat{r}_k=r^2 r_k \exp(-2\gamma_k h)+r^2[\Sigma]_{k,k}+ \tau\rho .
\]
In summary, the following must hold
\begin{equation}
\label{eqn:range2}
\rho\geq \max\left\{r_k,\frac{r\rho}{\tau}\exp(-2\gamma_k h)+\frac{r}{2\tau}E_k^u(1-\exp(-2\gamma_k h))\right\}
\end{equation}
except for at most $p$ different $k$. The quantity is easily computable, for example with the physical parameters \eqref{eqn:physics} while $J=50$ and $\sigma^o=10$, $p$ can be set as $6$ and $\rho=0.04$. Apparently, the existence of observations makes the effective dimension much smaller than the one estimated by the unfiltered covariance.

\subsection{Intermittent dynamical regime} 
One challenge that practical filters often face is that the dynamical coefficient $A_n$ is not always stable with spectral norm less than $1$. This is usually caused by the large scale chaotic dynamical regime transitions. One simple way of modeling this phenomenon in \eqref{sys:turb}, is letting $A_n$ be a Markov jump process \cite{MH12}, while maintaining the sub-block structure: $[A_n]_{\{k,-k\}^2}=[\lambda_{n}]_k[A]_{\{k,-k\}^2}.$ Here $\lambda_n$ is a Markov chain taking values in $\reals^{K+1}$. Then the system random instability can be modeled as the random fluctuation of $[\lambda_n]_k$, so that occasionally $\|[A_n]_{\{k,-k\}^2}\|>1$ for some $k$. 

Our framework is applicable to such scenarios in general, as we allow random system coefficients. The main difficulty would be the computation of the $\Rtilde_n$ and the verification of Assumption \ref{aspt:lowdim}, which in general require numerical methods. Our framework can also be generalized, so instead of  a constant ambient space  uncertainty level $\rho$, one can uses a stochastic sequence $\rho_n$. Section 5 of \cite{MT16uq} discuss this generalization. In this paper, we do not intend to generalize our framework to that level, as the analysis involved is already rather complicated.

On the other hand,  in many situations, the dynamical instability only occurs on the a small subset $I$ of Fourier modes. This is because when the wavenumbers are high,  the dissipation force is much stronger than the random environmental forcing. So for $k\in I^c$, $[A_n]_{\{k,-k\}}$ could remain of constant value like in \eqref{sys:turb}. See chapter 4 of \cite{MH12} for such an example, where the instability occurs only at a few mode with wavenumber less than $5$.  Then it suffices to include the subspace spanned by modes in $I$ in the instability subspace. This can be done with a slightly larger $p$. Since all verification discussed above, \eqref{eqn:range1} and \eqref{eqn:range2}, concern only of modes outside the instability subspace, they and also our framework still remain valid.

\section{Rigorous analysis for EnKF}
\label{sec:proof}
This section provides the main ingredients for the proof of Theorem \ref{thm:weak}. This is accomplished by a RMT result for the forecast covariance matrix, a Mahalanobis dissipation mechanism, a Kalman covariance comparison principle, and a Lyapunov function that connects the previous three. The detailed proof of the RMT result is delayed to Section \ref{sec:RMT}, as it is rather technical, long, and indirectly related to our main problem.  

There will be two filterations in our discussion. The first one contains all the information of system coefficients up to time $n$, and the initial covariance for the filters:
\[
\mathcal{F}^c_n=\sigma\{A_{k}, B_{k}, \Sigma_{k}, H_{k}, \sigma_{k}, k\leq n\}\vee \sigma\{R_0,  C_0, \Rtilde_0\}.
\]
 Noticeably,  the Kalman filters  have their covariance inside this filteration:
 $\sigma\{R_k, \Rtilde_k, k\leq n+1\}\subset\mathcal{F}^c_n.$  
We will use $\mathcal{F}^c=\vee_{n\geq 0} \mathcal{F}^c_n$ to denote all the information regarding the system coefficients through the entire time line. 

The second filteration contains all the information of system \eqref{sys:random} up to time $n$, $
\mathcal{F}_n=\mathcal{F}^c_n\vee \sigma\{\zeta_l, \xi_l, \zeta^{(k)}_l,\xi^{(k)}_l,l\leq n, k\leq K \}.$ 
We use $\E_n Z$, $\E_{\mathcal{F}}$ to denote the conditional expectation of a random variable $Z$ with respect to $\mathcal{F}_n$ or another fixed $\sigma$-field $\mathcal{F}$ respectively. $\Prob_n$ and $\Prob_{\mathcal{F}}$ denote the associated conditional probability.

\subsection{Concentration of samples}
\label{sec:sample}
The first mechanism we will rely on is that with a low effective dimension, the ensemble forecast covariance  $\Chat_{n+1}$ concentrates around its average. To describe this phenomenon, define the following sequences: 
\begin{equation}
\label{eqn:lambda}
\lambda_{n+1}=\inf\{\lambda\geq 1, \Chat^{\tau\rho}_{n+1}\preceq \lambda [rA_nC_nA_n^T+r\Sigma^+_n+r\tau\rho I]\}
\end{equation}
\begin{equation}
\label{eqn:mu}
\mu_{n+1}=\inf\{\mu\geq 1, [\Chat^{\tau\rho}_{n+1}]^{-1}\preceq  \mu [rA_nC_nA_n^T+r\Sigma^+_n +\tau\rho I]^{-1} \}. 
\end{equation}
So essentially the random matrix is sandwiched by its expected value multiplied by these ratios:
\[
\mu^{-1}_{n+1} [rA_nC_nA_n^T+r\Sigma^+_n +\tau\rho I]\preceq \Chat^{\tau\rho}_{n+1}\preceq \lambda_{n+1} [rA_nC_nA_n^T+r\Sigma^+_n+r\tau\rho I]. 
\]
Theorem \ref{thm:RMT} below indicates that $\lambda_{n}$ and $\mu_n$ are mostly of value close to one, as long as the sample size $K$ surpasses a constant multiple of the effective dimension $p$, specifically:
\begin{cor}
\label{cor:RMT}
Under the same conditions of Theorem \ref{thm:weak}, denote the following rare events
\[
\calU^\lambda_{n+1}=\{\lambda_{n+1}\geq \sqrt{r}\},\quad \calU^\mu_{n+1}=\{\mu_{n+1}\geq \sqrt{r}\}.
\]
There are constants $c_r,D_r>0$, such that 
\[
\Prob_n (\calU^\lambda_{n+1}\cup\calU^\mu_{n+1} )\leq \log( \|C_n\|+1)\exp(D_r p-c_rK).
\]
The tail of $\lambda_{n+1}$ can be bounded by an exponential one,
\[
\Prob_n(\lambda_{n+1}>8+t)\leq \exp(-c_r Kt ).
\]
Inside the rare event, for any fixed $M$, there is a constant $D_M$ such that the following bound  holds with $l\leq M$ 
\[
\E_n \unit_{\calU^\mu_{n+1}}\mu^l_{n+1}\leq D_{M}\|\Sigma_n^+/\rho\|^l\log (\|C_n\|+1)\exp(D_r p-c_rK).
\]
Recall that $\|\,\cdot\,\|$ denotes the $l_2$ operator norm of matrices. 
\end{cor}
\begin{proof}
We will apply Theorem \ref{thm:RMT} with 
 \[
 a_k=\sqrt{r}  A_n\Delta X_{n}^{(k)},\quad \xi_k=\sqrt{r}\xi^{(k)}_{n+1}.
 \] 
 so $Z=\Chat_{n+1}, \Sigma=r\Sigma_n^+, C=rA_nC_nA_n, D=C+r\Sigma^+_n$. We will  consider $\delta=\frac15(\sqrt{r}-1)$ and  two different $\rho$-s: $\rho_\lambda=(r-1)\rho\tau$ and $\rho_\mu=\rho\tau$. Theorem \ref{thm:RMT} is applicable here,  because after the projection $\bfP_n$, $\{a_1,\ldots, a_K\}$ spans a subspace of dimension at most $p$.  The union of the  rare events considered in this theorem is included by the one of Theorem \ref{thm:RMT}, since
 \[
 (\calU^\lambda_{n+1})^c=\{\Chat^{\tau\rho}_{n+1}\preceq (1+5\delta)(D+\rho_\lambda I)\}\supseteq 
 \{\Chat_{n+1}\preceq (1+5\delta)(D+r\tau \rho I)\},
 \]
 \[
 ( \calU^\mu_{n+1})^c=\{[\Chat^{\tau\rho}_{n+1}]^{-1}\preceq (1+5\delta)(D+\rho_\mu I)^{-1}\}.
 \]
 Notice that the condition number of $C+s I_d$ is $\|rs^{-1}A_nC_nA_n\|+1$. Also notice that  $\log (x+1)\leq \log x+1$ for all $x>0$, so with $\|A_n\|\leq M_A$, there is  a constant $D_{\rho,A,r}\geq 1$
  \[
  \log \|s^{-1}rA_n C_n A_n^T+I\| \leq \log (s^{-1}rM_A^2\|C_n\|+1)\leq D_{\rho,A,r}(\log\|C_n\|+1) 
  \]  
for both $s=(r-1)\rho\tau$ and $s=\rho\tau$. Therefore according to Theorem \ref{thm:RMT}, there are constants $D_\delta, c_\delta, D_M$ such that 
 \[
 \Prob_n(\calU^\lambda_{n+1}\cup\calU^\mu_{n+1})\leq  \exp(\log 4D_{\rho,A,r}+D_\delta p-c_\delta K)(1+\log \|C_n\|),\quad  \Prob_n(\lambda_n>8+t)<\exp(-c_\delta Kt),
 \]
 \[ 
 \E_n \unit_{\calU_{n+1}^\mu}\mu^l_n\leq \|\Sigma_n^+/\rho\|^lD_M \exp(\log 2D_{\rho,A,r}+D_\delta p-c_\delta K)(1+\|C_n\|),\quad l\leq M.
 \]
Since we can always pick $D_r$ such that  $\log 4D_{\rho,A,r}+D_\delta p\leq D_r p$ for all $p\geq 1$, we have our claims. 
 \end{proof}

\subsection{Covariance fidelity via Mahalanobis norm}
\label{sec:dissmaha}
The spectral projection is necessary for the posterior ensemble to be of rank $p$. The price we need to pay is that this procedure may decreases the ensemble covariance. Let $\rho_{n}$ be the $p+1$-th eigenvalue of $\Kalman_{n}(\Chat^{\tau\rho}_n)$, and $\chi_n=\max\{1, \rho_n/\rho\}$ measures the impact of the spectral projection step $\bfP_n$ over the posterior ensemble. This can be told by the following inequality
\begin{equation}
\label{eqn:nu}
\chi_{n}C^\rho_{n}\succeq\bfP_{n}(\Kalman_{n}(\Chat^{\tau\rho}_{n})-\rho I)\bfP_{n}+\chi_n\rho I\succeq
\bfP_{n}(\Kalman_{n}(\Chat^{\tau\rho}_{n})-\rho I)\bfP_{n}+\rho \bfP_{n}+ \rho_{n}(I-\bfP_{n})\succeq \Kalman_{n}(\Chat^{\tau\rho}_n). 
\end{equation}
This inequality can be verified by checking the eigenvectors of $\Kalman_{n}(\Chat^{\tau\rho}_n)$, which are the same as the one of $C^\rho_n$. The verification is straightforward if one divide the eigenvectors into the ones in the range of $\bfP_n$, and the ones in the null space. The Mahalanobis error dissipation can be formulated as below
\begin{lem}
\label{lem:contractensemble}
With our EnKF Algorithm \ref{alg:EnKF}, the filter error $e_n=\Xbar_n-X_n$ satisfies
\[
\E_n \|e_{n+1}\|^2_{C^\rho_{n+1}} \leq \frac{1}{r}\E_n \chi_{n+1}\mu_{n+1}\|e_n\|^2_{C^\rho_n}+ \E_n (\chi_{n+1}\mu_{n+1}+\chi_{n+1})d\quad a.s..
\]
\end{lem}
In our formulation above,  $\mu_{n+1}$ describes the fluctuation from random sampling, and $\chi_{n+1}$ describes  the possible deflation made by projection, as seen in \eqref{eqn:nu}.  In the classic Kalman filter setting \cite{RGYU99}, these sequences are simply ones.
\begin{proof}
The forecast estimator is $\overline{\Xhat}_{n+1}$, its error is $\hat{e}_{n+1}=\overline{\Xhat}_{n+1}-X_{n+1}$. The difference of the following two
\[
\overline{\Xhat}_{n+1}=A_n\Xbar_{n}+B_n, \quad X_{n+1}=A_n X_n+B_n+\xi_{n+1},
\]
yields $\hat{e}_{n+1}=A_n e_n-\xi_{n+1}$. Moreover, 
\[
\Xbar_{n+1}=(I-G_{n+1} H_{n}) \overline{\Xhat}_{n+1}+G_{n+1}Y_{n+1}=(I-G_{n+1} H_{n}) \overline{\Xhat}_n+G_{n+1}H_{n}X_{n+1}+G_{n+1}\zeta_{n+1},
\]
so 
\[
e_{n+1}=(I-G_{n+1} H_{n})\hat{e}_{n+1}+G_{n+1}\zeta_{n+1}=(I-G_{n+1} H_{n})A_ne_n-(I-G_{n+1} H_{n}) \xi_n+G_{n+1}\zeta_{n+1}.
\]
Because $\xi_{n+1}$ and $\zeta_{n+1}$ are distributed as $\mathcal{N}(0,\Sigma_n)$ and $\mathcal{N}(0,I)$, conditioned on $\mathcal{F}_n$, 
\begin{align}
\label{tmp:en}
&\E_n e_{n+1}^T[C^\rho_{n+1}]^{-1} e_{n+1}=\E_n e_n^TA^T_n(I-G_{n+1}H_{n})^T [C^\rho_{n+1}]^{-1} (I-G_{n+1}H_{n}) A_{n}e_n\\
\label{tmp:noise}
&+\E_n\xi_{n+1}^T (I-G_{n+1}H_n)^T[C^\rho_{n+1}]^{-1} (I-G_{n+1}H_{n}) \xi_{n+1}+\E_n \zeta^T_{n+1}G_{n+1}^T[C^\rho_{n+1}]^{-1}G_{n+1}\zeta_{n+1}. 
\end{align}
For the first part \eqref{tmp:en}, we claim that   
\begin{equation}
\label{tmp:pin}
A_{n}^T(I-G_{n+1}H_{n})^T [C^\rho_{n+1}]^{-1}(I-G_{n+1}H_{n}) A_{n}\preceq \frac{1}{r}\chi_{n+1}\mu_{n+1} [C_n^\rho]^{-1} .
\end{equation}
Because of \eqref{eqn:nu}, $\chi_{n+1}C^\rho_{n+1}\succeq \Kalman_{n}(\Chat^{\tau\rho}_{n+1}) \succeq(I-G_{n+1}H_{n})\Chat^{\tau\rho}_{n+1} (I-G_{n+1}H_n)^T.$ Moreover $(I-G_{n+1}H_{n})=(I+\Chat^{\tau\rho}_{n+1}H_n^TH_n)^{-1}$ is clearly invertible. The inversion of the inequality above reads
\begin{equation}
\label{tmp:Kalcor}
(I-G_{n+1}H_{n})^T[C^\rho_{n+1}]^{-1}(I-G_{n+1}H_{n})\preceq\chi_{n+1} [\Chat^{\tau\rho}_{n+1}]^{-1}.
\end{equation}
Next, recall \eqref{eqn:sampledominate} we have
\begin{equation}
\label{tmp:expect}
rA_n C_n A_n^T+r\Sigma^+_n+ \rho \tau I\succeq
rA_n C_n^\rho A_n^T.
\end{equation}
By the definition  of $\mu_{n+1}$, $ [\Chat^{\tau\rho}_{n+1}]^{-1}\preceq \mu_{n+1}[rA_n C_n A_n^T+r\Sigma^+_n+ \rho \tau I]^{-1}$, so
\[
 A^T_n(I-G_{n+1}H_{n})^T[C_{n+1}^\rho]^{-1} (I-G_{n+1}H_{n})A_n\preceq \frac{1}{r}\mu_{n+1} \chi_{n+1} A_n [A_n C^\rho_n A_n^T]^{-1}A_n^T,
\]
which by  Lemma \ref{lem:matrix} leads to \eqref{tmp:pin}. To deal with \eqref{tmp:noise}, we use the identity $a^T A a=\text{tr}(A a a^T)$ and the conditional distributions of the system noises,  
\begin{align*}
&\E_n \xi_{n+1}^T (I-G_{n+1}H_n)^T[C^\rho_{n+1}]^{-1} (I-G_{n+1}H_{n}) \xi_{n+1}+\zeta^T_{n+1}G_{n+1}^T[C^\rho_{n+1}]^{-1}G_{n+1}\zeta_{n+1}\\
&=\E_n\text{tr} [(I-G_{n+1}H_{n})[C^\rho_{n+1}]^{-1} (I-G_{n+1}H_n)^T\Sigma_n+ G_{n+1}G_{n+1}^T[C^\rho_{n+1}]^{-1}].
\end{align*}
Note that by \eqref{tmp:Kalcor} and \eqref{tmp:expect}, and $\Chat^{\tau\rho}_{n+1}\succeq r\Sigma^+_n+\tau\rho I$
\[
(I-G_{n+1}H_{n})[C^\rho_{n+1}]^{-1} (I-G_{n+1}H_n)^T
\preceq \chi_{n+1} [\Chat^{\tau\rho}_{n+1}]^{-1}\preceq \chi_{n+1}\mu_{n+1} [r\Sigma^+_n+\tau\rho I ]^{-1} ,
\]
By Lemma \ref{lem:matrix}, and $r\Sigma_n^++\tau\rho I\succeq r\Sigma_n$, 
  \[
  \text{tr} [(I-G_{n+1}H_{n})[C^\rho_{n+1}]^{-1} (I-G_{n+1}H_n)^T\Sigma_n]\leq \frac{1}{r}\chi_{n+1}\mu_{n+1}d.
  \] 
Also notice that
\[
\Kalman_{n}(\Chat^{\tau\rho}_{n+1})=(I-G_{n+1}H_n)\Chat^{\tau\rho}_{n+1}(I-G_{n+1}H_n)^T+G_{n+1}G_{n+1}^T\succeq G_{n+1}G_{n+1}^T. 
\]
Then by $\chi_{n+1}C^\rho_{n+1}\succeq  \Kalman_{n}(\Chat^{\tau\rho}_{n+1})$, $\text{tr}( G_{n+1}G_{n+1}^T[C^\rho_{n+1}]^{-1})\leq d\chi_{n+1}.$
As a sum, we have shown our claim. 
\end{proof}

\subsection{Covariance control via comparison}
\label{sec:corcontrol}
The second mechanism we will exploit is the comparison principle of the Riccati equation. The general idea is to compare the ensemble covariance $C_n$ with a solution of the Riccati equation \eqref{sys:optimalr}. This can be done in two fashions. The first one is through expectation 
\begin{lem}
\label{lem:corcoarse}
For all $n\geq 0$, $\E_{\mathcal{F}^c} C_n\preceq \Riccati_{0,n}(C_0)$ a.s.
\end{lem}
\begin{proof}
Let $R'_n=\Riccati_{k,n}(C_k)$. We will prove our claim by induction. Suppose that $\E_{\mathcal{F}^c} C_n\preceq R'_n a.s.$, then since $\xi^{(k)}_{n+1}$ is $\mathcal{N}(0,\Sigma_n^+)$ conditioned on $\mathcal{F}^c\vee \mathcal{F}_n$ and 
\[
\E_{\mathcal{F}^c}\Chat^{\tau\rho}_{n+1}=
\E_{\mathcal{F}^c}\E_{\mathcal{F}^c\vee\mathcal{F}_n}\Chat^{\tau\rho}_{n+1}
= rA_n (\E_{\mathcal{F}^c}C_n) A_n^T+r\Sigma_n^++\tau\rho I_d\preceq rA_nR'_nA_n^T+r\Sigma'_n\preceq \Rhat'_{n+1}.
\]
By Lemma \ref{lem:Kalconcave}, $\Kalman_{n}$ is concave and monotone. By Jensen's inequality
\[
\E_{\mathcal{F}^c}C_{n+1}\preceq \E_{\mathcal{F}^c}\Kalman_{n}(\Chat^{\tau\rho}_{n+1})\preceq \Kalman_{n} (\E_{\mathcal{F}^c}\Chat^{\tau\rho}_{n+1})\preceq \Kalman_{n}(\Rhat'_{n+1})=R'_{n+1},\quad a.s.
\] 
\end{proof}
This result explains why covariance inflation is necessary: the random sampling underestimates the covariance on average. On the other hand, the a-priori estimates it provide on $C_n$ is rather limited. For example, $\E_{\mathcal{F}^c}\|C_n\|$ cannot be obtained solely from Lemma \ref{lem:corcoarse}. In order to do a more delicate analysis, we  consider the quotient ratio between $C_n$ and  $\Rtilde_n$ as in Assumption \ref{aspt:lowdim}. Define the following sequence
\begin{equation}
\label{eqn:nudef}
\nu_{n}=\inf\{\nu\geq 1, C_n\preceq\nu \Rtilde_n\}.
\end{equation}
By Lidskii's theorem 6.10 of \cite{Kat82}, $C_n\preceq \nu_n\Rtilde_n$ indicates the ordered eigenvalues of $C_n$ is dominated by the ordered eigenvalues of $\nu_n\Rtilde_n$. Then by Assumption \ref{aspt:lowdim}, the $p+1$-th eigenvalue of $C_n$ is at most $\nu_n\rho$. Therefore sequence $\nu_n$ dominates sequence $\chi_n$. It is of this reason, we will replace $\chi_n$ with its upper bound $\nu_n$ in the following discussion.

 As a matter of fact, we can compare our EnKF with any solution of \eqref{sys:optimalr}, and find the following recursion formula 
\begin{lem}
\label{lem:corfine}
Fix a time $k$, and a covariance sequence $R'_n=\Riccati_{k,n}(R'_k)$. Let $\nu'_n=\inf\{\nu\geq 1, C_n\preceq \nu R'_n\}.$ Then  
\[
\nu'_{n+1}\leq \min\{1, \nu'_n \lambda_{n+1}/r\}.
\] 
Moreover,  $\Kalman_{n}(\Chat^{\tau\rho}_{n+1})\preceq \nu_{n+1}'\Rhat'_{n+1}$. 
\end{lem}
\begin{proof}
We will prove our claim by induction. Suppose that $ C_n\preceq \nu'_n R'_n$,
\[
\lambda_{n+1}\nu'_n  \Rhat'_{n+1}=\lambda_{n+1}\nu'_n r^2(A_n R'_n A_n^T+\Sigma^+_n+ \tau\rho I_d)
\succeq r \lambda_{n+1}( rA_nC_nA_n+r\Sigma^+_n+r\tau\rho I_d)\succeq   r\Chat^{\tau\rho}_{n+1},
\]
where we used that the definition of $\lambda_{n+1}$. Then by the concavity and monotonicity  of $\Kalman_{n}$, Lemmas \ref{lem:Kalconcave} and \eqref{eqn:objcor}
\[
C_{n+1}\preceq \Kalman_{n} (\Chat_{n+1}^{\tau\rho})\preceq \Kalman_{n}(\nu'_{n+1}\Rhat'_{n+1})
\preceq \nu'_{n+1}\Kalman_{n}(\Rhat'_{n+1})= \nu'_{n+1}R'_{n+1}. 
\]
\end{proof}
The uniform observability condition guarantees an upper bound for our covariance. 
\begin{lem}
\label{lem:boucontract}
Under the uniform observability  Assumption \ref{aspt:obcon}, given any matrix $C_k$, there is a $D_\Riccati$ such that
\[
\|\Riccati_{k,m+k}(C_k) \Rtilde_{m+k}^{-1}\|\leq D_\Riccati.
\]
\end{lem}
\begin{proof}
This is proved by proposition 19 of \cite{MT16uq} (proposition 6.1 for the arXiv version). By taking no information for forecast at time $k$ (that is formally taking $\Rhat_k^{-1}=0$), we have
\[
\Riccati_{k,m+k}(C_k)\preceq \sum^{k+m}_{j=k+1}Q^j_{k,k+m}\Sigma'_j (Q^j_{k,k+m})^T,\quad Q_{k,m+k}^j=r^{j-k}A_{k,k+m}\mathcal{O}^{-1}_{k,k+m}\mathcal{O}_{j,m+k}A^{-1}_{j,m+k}. 
\]
 Then because $\mathcal{O}_{k,k+m}\succeq \mathcal{O}_{j,m+k}$, so $\|\mathcal{O}^{-1}_{k,k+m}\mathcal{O}_{j,m+k}\|\leq 1$, and recall the bounds in Assumption \ref{aspt:obcon}
\[
\|\Sigma_n'\|\leq r^2\|\Sigma_n\|+r^2\rho\|A_n\|^2+r^2\tau\rho 
\leq r^2(D_\Sigma+\rho D_A^2+\tau\rho)
\]
 \[
 \|Q_{k,m+k}^j\|\leq r^m\|A_{k,k+m}\|\|\mathcal{O}^{-1}_{k,k+m}\mathcal{O}_{j,m+k}\|\|A^{-1}_{j,m+k}\| \leq r^m D_A^{2m}. 
 \]
 Therefore $\|\Riccati_{k,m+k}(C_k)\|\preceq r^{2m+2} m D_A^{4m}(D_\Sigma+\rho D_A^2+\tau\rho).$ Our claim follows as $\|\Rtilde^{-1}_{m+k}\|\leq D_R$.
\end{proof}

\subsection{A Lyapunov function}
In the view of Lemma \ref{lem:contractensemble}, the Mahalanobis error is not dissipating by itself, due to the fluctuation of the sampling effect and truncation errors. Sections  \ref{sec:sample} and \ref{sec:corcontrol} imply these effects are controllable. In this section, we show that they are combined under a Lyapunov function, namely the following 
\begin{equation}
\label{eqn:psiphi}
\phi(\nu)=\exp(D_\psi \log^3 \nu),\quad\psi(\nu)=\phi(\nu)\nu^m(1+a_r \nu). 
\end{equation}
$D_\psi$ is a large constant and $a_r$ is close to $0$, their values will be fixed during our discussion. Before we prove Theorem \ref{thm:weak}, we need two components. We will assume Assumptions \ref{aspt:lowdim} and \ref{aspt:obcon} throughout our discussion in this subsection.  The first component iterates Lemma \ref{lem:contractensemble} inside a time interval of size $m$:
\begin{lem}
\label{lem:mtimesmaha}
For any $b_r>0$, there is a constant $D_b$ such that if $K>D_b p$, and $n\leq m$ 
\[
\E_0 \|e_n\|^2_{C_n^\rho}\leq (r^{-\frac{n}{2}}+b_r(1+\log \nu_0))\nu^{n}_0 \|e_0\|^2_{C_0^\rho} + 2(1+b_r(1+\log\nu_0))n\nu_0^n\sqrt{d}. 
\]
\end{lem}
\begin{proof}
Since $\nu_n$ dominates $\xi_n$, we can iterate Lemma \ref{lem:contractensemble} $n$ times, and find that
\begin{equation}
\label{tmp:emexp}
\E_0 \|e_n\|^2_{C_n^\rho}\leq \left(\E_0 \prod_{k=1}^n \nu_k \mu_k/r\right) \|e_0\|_{C_0^\rho}^2 +
\sqrt{d}\E_0 \sum_{j=1}^n(\nu_j\mu_j/r+\nu_j)\prod_{k=j+1}^n \nu_k \mu_k/r.
\end{equation}
We will deal with the linear coefficient $\E_0 \prod_{k=1}^n \nu_k \mu_k/r$ first.  Consider the following rare events:
\begin{equation}
\label{eqn:rarek}
\calU_k=\{\mu_k\geq \sqrt{r}\text{ or }\lambda_k\geq \sqrt{r}\},\quad k=1,\ldots, n,
\end{equation}
and $\calU=\cup_{k=1}^n \calU_k$. Outside of $\calU$, $\lambda_k\leq \sqrt{r},\mu_k\leq \sqrt{r}$, so by Lemma \ref{lem:corfine}, $\nu_k\leq \nu_0$ and
\[
\prod_{k=1}^n \nu_k \mu_k/r\leq \nu_0 ^n r^{-\frac{n}{2}}. 
\]
This indicates that 
\[
\E_0 \prod_{k=1}^n \nu_k \mu_k/r\leq \nu_0 ^n r^{-\frac{n}{2}}+\E_0 \unit_{\calU} \prod_{k=1}^n\nu_k\mu_k.
\]
Note that since $\lambda_k\geq 1$, Lemma \ref{lem:corfine} indicates $\nu_k\leq \nu_0\prod_{j=1}^k\lambda_j$ a.s.. So the rare event part in the right hand side above can be bounded as follows:
\begin{align}
\notag
\E_0 \unit_{\calU} \prod_{k=1}^n\nu_k\mu_k\leq 
\nu_0^n\E_0 \unit_{\calU} &\prod_{k=1}^n \lambda_k^{n-k} \prod_{k=1}^n\unit_{\calU}\mu_k
\leq \nu_0^n[\Prob_0(\mathcal{\calU})]^{1/3}\left[\E_0 \prod_{k=1}^n \lambda^{3(n-k)}_k  \right]^{1/3}\left[\E_0 \unit_{\calU}\prod_{k=1}^n \mu^3_k\right]^{1/3}\\
\label{tmp:rareuk}
&\leq \nu_0^n[\Prob_0(\mathcal{\calU})]^{1/3}\prod_{k=1}^n\left[\E_0  \lambda^{3n(n-k)}_k \right]^{1/3n}\prod_{k=1}^n \left[\E_0 \unit_{\calU_k}\mu^{3n}_k\right]^{1/3n}
\end{align}
Because $n\leq m$, by Corollary \ref{cor:RMT} there are constant $c_r, D_r$ such that $\Prob_0 (\calU)\leq \sum_k\Prob_0 (\calU_k)$ with
\[
\Prob_0(\calU_k)= \E_0 \Prob_{k-1}(\calU_k)\leq   \exp(D_r p-c_rK)\E_0\log( \|C_{k-1}\| +1 )\\
\]
To continue, by Jensen's inequality, the concavity of $\log$, and Lemma \ref{lem:corcoarse}
\begin{align*}
\E_0  \log (\|C_{k-1}\|+1)\leq &\E_0\log  (\text{tr}(C_{k-1})+2)\leq \log (\text{tr}\E_0 C_{k-1}+1)\\
&\leq  \log (\text{tr}(\Riccati_{0,k-1}(C_0))+1)
\leq \log p+\log(\|\Riccati_{0,k-1}(C_0)\|+\tfrac1p).
\end{align*}
Note that $\Riccati_{0,k}(C_0)$ is the Kalman filter covariance of system \eqref{sys:changed} with $X_0\sim \mathcal{N}(0,C_0)$. Therefore it is dominated by the unfilter covariance 
\[
\Riccati_{0,k}(C_0)\preceq V_k=A_{k,0}C_0A_{k,0}^T+\sum_{j=1}^k A_{k,j}\Sigma_j' A_{k,j}.
\]
By the bounds of $\|A_{k,j}\|$ and $\|\Sigma_j\|$ for $k\leq m$, and that $C_0\preceq \nu_0\Rtilde_0\Rightarrow C_0\leq \nu_0\|\Rtilde_0\|$, there is a constant $D_m$ such that 
\begin{equation}
\label{eqn:probU}
\Prob_0 (\calU)\leq m\sup_{k\leq m}( \log p+\log(\|\Riccati_{0,k-1}(C_0)\|+\tfrac1p))\leq D_m (\log\nu_0+1).
\end{equation}
By Corollary \ref{cor:RMT}, $D_m$ can be properly enlarged so that $\E_0 \lambda_k^{3n(n-k)}\leq D_m$ for all $n\leq m$, and 
\[
\E_0 \unit_{\calU_k}\mu_k^{3n}\leq D_m(\log \nu_0+1)\exp(D_rp-c_rK).
\]
Therefore, $\eqref{tmp:rareuk}\leq \nu_0^n\exp(\tfrac{2}{3}D_rp-\tfrac23c_rK)D_m(\log \nu_0+1)$. In summary, if 
\[
K-D_b p:=K-\frac{D_r}{r_r}p
\] is large enough, 
\[
\E_0\prod_{k=1}^n \nu_k \mu_k/r\leq \nu_0^n (r^{-\frac{n}{2}}+b_r(1+\log \nu_0)).
\]
The constant terms in \eqref{tmp:emexp} can be bounded in a similar fashion. Note that outside the rare event $\calU$, 
$(\nu_j\mu_j/r+\nu_j)\prod_{k=j+1}^n \nu_k \mu_k/r\leq 2\nu_0^m$. And inside the rare event, we can bound it exactly like in \eqref{tmp:rareuk}, but with fewer terms, so 
\[
\E_0 (\nu_j\mu_j/r+\nu_j)\prod_{k=j+1}^n \nu_k \mu_k/r
\leq 2\nu_0^n(1+b_r(1+\log \nu_0)). 
\]
Using each term in \eqref{tmp:emexp} with the corresponding upper bound above yields our claim. 
\end{proof}

The second component shows that $\nu_n$ is a very stable sequence, it indicates that $C_n$ is dominated by $\Rtilde_n$ for most of the times. But first we have a purely computational verification:
\begin{lem}
\label{lem:lambdapsi}
For any fixed $c_r>0$, there is a constant $v_0$, such that for any $D_\psi$, there exists a $K$, such that if $Z$ has exponential distribution with parameter $c_r(K-1)$, then
\[
\E \exp(D_\psi\log^3(8+Z))\leq \exp(D_\psi v_0).
\]
\end{lem}
\begin{proof}
Denote $c_K=c_r(K-1)$,
\[
\E \exp(D_\psi\log^3(8+Z))=\int^\infty_0 c_K\exp(D_\psi \log^3 (8+z))\exp(-c_Kz)dz=\int^\infty_0\exp(D_\psi\log^3(8+z/c_K)-z)dz. 
\]
Let $u=\log(8+z/c_K)\geq 0$, then by $\exp(u)\geq \frac{1}{6}u^3$,
\[
D_\psi\log^3(8+z/c_K)-\frac{1}{2}z=D_\psi u^3-\frac{1}{2}c_K(\exp u-8)\leq (D_\psi-\frac{1}{12}c_K) u^3+4c_K. 
\]
So if we let $c_K=12D_\psi$, then 
\[
\int^\infty_0\exp(D_\psi\log^3(8+z/c_K)-z)dz\leq \int^\infty_0\exp(48D_\psi-\frac{1}{2}z)dz=2\exp(48D_\psi). 
\]
So clearly we can find our $v_0$. 
\end{proof}
The second component is the following. 
\begin{lem}
\label{lem:phi}
For a sufficiently small $a_r\in (0,1)$, there are constants $D_\psi, D_b\geq 1$ such that with $\phi(\nu)=\exp(D_\psi \log^3 \nu)$, and $K-D_b p$ being sufficiently large,
\begin{equation}
\label{tmp:psi}
\E_0\psi(\nu_m) =\E_0\phi(\nu_{m})\nu_m^m(1+a_r\nu_m) \leq (1+2a_r)\phi(\nu_0). 
\end{equation}
Also with constant $\gamma_\phi=\exp(-\frac18 D_\psi \log^3 r)<1$ and any $k\leq m+1$
\begin{equation}
\label{tmp:phi}
\E_0\nu_m^k\phi(\nu_{m}) \leq \gamma_\phi \nu_0^k\phi(\nu_0)+1+a_r. 
 \end{equation}
\end{lem}
\begin{proof}
First, let us show \eqref{tmp:psi} assuming $\nu_0$ is very large. Let $R'_k=\mathcal{R}_{0,k}(C_0)$, and 
\[
\nu'_k=\inf\{\nu\geq 1, C_k\preceq \nu R_k' \}. 
\]
Then $\nu'_0=1$, and by Lemma \ref{lem:corfine}, $\nu'_{k+1}\leq \max\{1, \nu'_k \lambda_{k+1}/r\}$. $\nu'_m$ gives us a bound for $\nu_m$, because 
\[
\nu_m\leq \nu_m'\left\|R'_m \Rtilde_m^{-1}\right\|, 
\]
where by Lemma \ref{lem:boucontract}, there is a constant $D_\Riccati$ such that, $\left\|R'_m \Rtilde_m^{-1}\right\|\leq D_\Riccati$. Since $\lambda_k\geq 1$, we find
\[
\log \nu_m\leq \log\nu_m'+\log D_\Riccati\leq \sum_{k=1}^m \log^+(\lambda_k /r)+\log D_\Riccati\leq \sum_{k=1}^m \log \lambda_k+\log D_\Riccati,
\]
with $\log^+=\max\{\log, 0\}$. Plug this upper bound into 
\[
\log (\phi(\nu_m) \nu_m^m(a_r\nu_m+1))\leq \log (2\phi(\nu_m) \nu_m^{m+1})=D_\psi\log^3 \nu_m+(m+1)\log\nu_m+ \log2.
\] 
By Young's inequality, there exists a constant $D_1$ independent of $D_\psi$  such that  
\[
\E_0 \phi(\nu_m) \nu_m^m(a_r\nu_m+1)\leq \E_0\exp\left(D_\psi D_1+D_\psi D_1 \sum_{k=1}^m \log^3 \lambda_k\right).
\]
By Corollary \ref{cor:RMT}, we can find a sequence of independent, exponential with parameter $c_r(K-1)$, random variables $Z_k$ such that  $\lambda_k\leq 8+Z_k$ a.s.. Then because the function above is increasing with $\lambda_k$
 Then 
\[
\E_0\exp\left(D_\psi D_1+D_\psi D_1 \sum_{k=1}^m \log^3 \lambda_k\right)\leq 
\exp(D_\psi D_1)\prod_{k=1}^m\E(D_\psi D_1 \log^3 (8+Z_k)).
\] 
By Lemma \ref{lem:lambdapsi}, there is a $\Pi$ such that if $\nu_0>\Pi$, then for any fixed $D_\psi$, for $K$ larger than a constant $K_\psi$
\[
\E_0\psi(\nu_m)\leq (1+2a_r)\phi(\nu_0),
\]
 which completes \eqref{tmp:psi}. As for \eqref{tmp:phi}, simply note that in the previous step we have built an upper bound for $\phi(\nu_m) \nu_m^{m+1}$, while  $(1+2a_r)\phi(\nu_0)\leq \gamma_\phi\phi(\nu_0)\nu_0^k+1$ for a sufficiently large $\Pi$ and $\nu_0\geq \Pi$. 

Next we do a more delicate bound with $\nu_0\leq \Pi$ for \eqref{tmp:psi}. This time, we simply apply Lemma \ref{lem:corfine} to $R'_k=\Rtilde_k$, so 
\begin{equation}
\label{tmp:mark1}
\nu_{n+1}\leq \max\{1, \nu_n \lambda_{n+1}/r\}.
\end{equation}
Recall the rare events $\calU=\cup_{k=1}^m\calU_k$ with $\calU_k$ given by \eqref{eqn:rarek}.  First, let us assume $\nu_0\geq \sqrt{r}$, then within $\calU^c$, $\nu_m\leq\nu_1\leq \nu_0/\sqrt{r}$, 
\begin{align*}
\label{tmp:Dpsi}
 &D_\psi \log^3 \nu_m+k\log \nu_m+\log(1+a_r \nu_m)\\
 &\leq D_\psi (\log \nu_0-\tfrac{1}{2}\log r)^3+ k(\log \nu_0-\tfrac{1}{2}\log r)+\log(1+a_r\nu_0/\sqrt{r}).
\end{align*}
Note that for sufficiently small $a_r$, $1+a_r\nu_0/\sqrt{r}\leq \max\{r^{m/2}, \nu_0/\sqrt{r}\}$ for all $\nu_0\geq 0$. So for any fixed $D_\psi\geq 1$, one can verify that 
\[
D_\psi \log^3 \nu_m+k\log \nu_m\leq D_\psi \log^3 \nu_0+k\log \nu_0+\log \gamma_\phi.
\]
Moreover  for a sufficiently large $D_\psi$,
\[
D_\psi \log^3 \nu_m+m\log \nu_m+\log(1+a_r \nu_m)\leq D_\psi \log^3 \nu_0. 
\]
Else if $\nu_0\leq \sqrt{r}$, then outside $\calU$, $\nu_m=1$, which makes $D_\psi \log^3 \nu_m+k\log \nu_m=0$,
 \[
 D_\psi \log^3 \nu_m+m\log \nu_m+\log(1+a_r\nu_m)=\log(1+a_r)\leq D_\psi \log^3 \nu_0+\log(1+a_r).
 \]
 So for all $\nu_0\geq 1$, we can show that  
 \[
 \E_0\unit_{\calU^c}\phi(\nu_m)\nu_m^m(1+a_r\nu_m)\leq (1+a_r)\phi(\nu_0),\quad
 \E_0\unit_{\calU^c}\phi(\nu_m)\nu_m^k\leq \gamma_\phi\phi(\nu_0)\nu_0^k+1.
 \]
Therefore 
\[
\E_0 \psi(\nu_m)\leq  (1+a_r)\phi(\nu_0)+\E_0 \unit_{\calU} \psi(\nu_m),\quad
\E_0 \phi(\nu_m)\nu_m^k\leq  \gamma_\phi\phi(\nu_0)\nu_0^k+1+\E_0 \unit_{\calU} \phi(\nu_m)\nu_m^k.
\]
It remains to bound the rare event part. Since $k\leq m$, we want  to show
 $\E_0  \unit_{\calU} \phi(\nu_m)\nu_m^k\leq \E_0 \unit_{\calU} \psi(\nu_m)\nu_m \leq a_r$. Apply the Cauchy Schwarz inequality, we find 
\[
\E_0 \unit_{\calU} \psi(\nu_m)\nu_m\leq \Prob_0(\calU)^{1/2} \E_0 \psi^2(\nu_m)\nu^2_m.
\]
Since $\lambda_k\geq 1$,  $\nu_m\leq \nu_0 \prod_{k=1}^m \lambda_k$, and $(1+a_r \nu_m)\leq 2\nu_m$,
\[
\E_0 \psi^2(\nu_m)\nu_m^2\leq 4\E_0 \exp\left(2D_\psi\left(\log \nu_0 +\sum_{k=1}^m \log \lambda_k\right)^3+2(m+2)\log \nu_0 +2(m+2)\sum_{k=1}^m \log \lambda_k\right).
\]
Recall that we can bound $\lambda_k$ by a sequence of exponential random variables $x_k$, so the quantity above is bounded by 
\begin{equation}
\label{tmp:mark2}
4\E_0 \exp\left(2D_\psi\left(\log \nu_0 +\sum_{k=1}^m \log (8+x_k)\right)^3+2(m+2)\log \nu_0 +2(m+2)\sum_{k=1}^m \log (8+x_k)\right).
\end{equation}
Following a similar computational result like the one of Lemma \ref{lem:lambdapsi}, there is a constant $D'_\psi$ that bounds the right hand side for all $
\nu_0\leq \Pi$. Therefore, by upper bounds of $\Prob_0(\calU)$ as in \eqref{eqn:probU}, there is a $D''_\psi$
\[
\Prob_0(\calU)^{1/2} \sqrt{\E_0 \psi^2(\nu_m)\nu_m^2}
\leq \exp(\tfrac12D_rp-\tfrac12c_rK) D''_\psi. 
\]
When $c_rK-D_rp$ is sufficiently large, this can be further bounded by $a_r$. 
\end{proof}
\begin{lem}
\label{lem:psikk}
For any $1\leq k\leq m$, the following holds if $K-D_bp$ is sufficiently large
\[
\E_0\psi(\nu_k) \leq 1+\psi^4 (\nu_0). 
\]
\end{lem}
\begin{proof}
The proof is the same as Lemma \ref{lem:mtimesmaha}, starting from \eqref{tmp:mark1} to \eqref{tmp:mark2}, which holds for general $\nu_0$ and we replace $m$ with $k$. The only difference is that \eqref{tmp:mark2} is no longer bounded by a constant, as we no longer have $\nu_0\leq \Pi$. So instead, we bound \eqref{tmp:mark2} by the following using H\"{o}lder's inequality $(a+b)^3\leq 4a^3+4b^3$
\[
4 \exp\left(8D_\psi\log^3 \nu_0+2(k+2)\log \nu_0 \right)\E\exp\left(8D_\psi\sum_{j=1}^k \log^3 (8+x_j)+2(k+2)\sum_{j=1}^k\log (8+x_j)\right).
\]
The quantity above can  further be bounded by $\psi^8(\nu_0)D_\psi''$ for a constant $D_\psi''$ using computations like in Lemma \ref{lem:lambdapsi}. Hence 
\[
\E_0 \unit_{\calU} \psi(\nu_k)\leq \sqrt{\Prob_0(\calU) \E_0 \psi^2(\nu_k)}\leq \exp(\tfrac12D_rp-\tfrac12c_rK) \sqrt{D'_\psi}\psi^4(\nu_0).
\]
So if $K-D_bp$ is sufficiently large, $\E_0\psi(\nu_k)\leq (1+a_r)\phi(\nu_0)+\E_0 \unit_{\calU} \psi(\nu_k)$ is further bounded by $1+\psi^4(\nu_0)$ for  $a_r\leq 1/2$.  
\end{proof}
\subsection{Conclusions from previous results}
Once we combine all previous arguments, we reach a stronger version of Theorem \ref{thm:weak}:
\begin{thm}
\label{thm:strong}
Suppose system \eqref{sys:random} is uniformly observable as in Assumption \ref{aspt:obcon} and has intrinsic dimension $p$ as in Assumption \ref{aspt:lowdim}, then for any fixed $a_r>0$ sufficiently small and any fixed $D_\psi$ sufficiently large, there exists a $\bfD$, such that if $K-\bfD p$ is sufficiently large, and
\[
\phi(\nu)=\exp(D_\psi \log^3\nu ),\quad\psi(\nu)=\phi(\nu)\nu^m(1+a_r \nu),\quad \gamma_\phi=\exp(- \frac18 D_\psi\log^3r),
\]
then for any $k\in [0,m)$
\[
\E_0 \sqrt{\psi(\nu_{nm+k)}} \|e_{nm+k}\|_{C^\rho_{nm+k}} <r^{-\frac{nm}{6}}\sqrt{(\psi(\nu_0)+\psi^3(\nu_0))(\|e_0\|^2_{C^\rho_0}+2k)}
+M_n\sqrt{d}.
\]
\[
\E_0 \psi(\nu_{nm+k})\leq \gamma_\phi^n (\psi(\nu_0)+\psi^5(\nu_0))+\frac{(1+a_r)^2}{1-\gamma_\phi}.
\]
The sequence $M_n$ is given by the following with,
\[
M_n=\sum_{k=0}^n r^{\frac{(k-n)m}{6}}\sqrt{\gamma^k_\phi (1+\psi^4(\nu_0))+\frac{1+2a_r}{1-\gamma_\phi}}. 
\]
With $n\to \infty$, $M_n$ converges to a constant $\sqrt{\frac{1+2a_r}{(1-\gamma_\phi)(1-r^{-\frac{m}{6}})^2}}$. 
\end{thm}
Before we show our proof, it is not difficult to go from Theorem \ref{thm:strong} to Theorem \ref{thm:weak}, as $F(C_0)$ can be chosen as 
\[F(C_0)=r^{\tfrac{m}{6}}\sqrt{\psi(\nu_0)+\psi^5(\nu_0)}, 
\]
where $\nu_0$ is defined as in \eqref{eqn:nudef}
 and  $M_n$ in both theorems are the same. So the bound for $\E \|e_n\|_{C_n^\rho}$ comes immediately from $\psi\geq 1$. Then because $\psi(\nu)\geq \nu$,  $\Rtilde^\rho_n\preceq \nu_n C_n^\rho$ and $\|\Rtilde_n^\rho\|\leq D_R+\rho$
 \[
\E|e_n|\leq  \sqrt{D_R+\rho}\E\|e_n\|_{\Rtilde^\rho_n}\leq \sqrt{D_R+\rho}\E\sqrt{\nu_n}\|e_n\|_{\Rtilde^\rho_n}\leq \E\sqrt{\psi(\nu_n)}\|e_n\|_{\Rtilde^\rho_n},
 \]
 we have the bound of $\E|e_n|$. The dominance of $C_n$ over  $\Rtilde_n$ \eqref{eqn:cnbound} can be derived from the bound of $\E \psi(\nu_{nm+k})$, as in $D_\phi$ can be chosen as arbitrarily large, and $a_r$ arbitrarily close to $0$. 
 
\begin{proof}[Proof of Theorem \ref{thm:strong}]
We will pick parameters so that Lemmas \ref{lem:mtimesmaha}-\ref{lem:psikk} all hold.  First, we consider the case with $k=0$. By Cauchy Schwarz, 
\begin{align*}
&(\E_0\sqrt{\psi(\nu_m)} \|e_m\|_{C_m^\rho})^2
\leq \E_0 \psi(\nu_m)\E_0 \|e_m\|^2_{C^\rho_m}\\
&\leq (1+2a_r)(r^{-\frac{m}{2}}+b_r(1+\log\nu_0))\nu_0^m\phi(\nu_0)\|e_0\|^2_{C^\rho_0}
+2m(1+2a_r)(1+b_r(1+\log\nu_0))\nu_0^m\phi(\nu_0).
\end{align*}
For sufficiently small $a_r, b_r$, because $\log\nu_0\leq \nu_0-1$
\[
(1+2a_r)(r^{-\frac{m}{2}}+b_r(1+\log\nu_0))\leq r^{-\tfrac m 3}(1+a_r\nu_0)^2,
\quad (1+2a_r)(1+b_r(1+\log\nu_0))\leq 2(1+a_r\nu_0)^2. 
\] 
Then by $\sqrt{a+b}\leq \sqrt{a}+\sqrt{b}$, 
\[
\E_0\sqrt{\psi(\nu_m)} \|e_m\|_{C_m^\rho}
\leq r^{-\frac{m}{6}}\sqrt{\psi(\nu_0)}\|e_0\|_{C_0^\rho}+\sqrt{2m\psi(\nu_0)}. 
\]
Using Markov property, we can iterate this inequality $n$ times using Gronwall's inequality
\[
\E_0 \sqrt{\psi(\nu_{nm})} \|e_{nm}\|_{C_{nm}^\rho}
\leq r^{-\frac{mn}{6}}\sqrt{\psi(\nu_0)}\|e_0\|_{C_0^\rho}+\sum_{j=1}^n r^{-\frac{m(n-j)}{6}} \E_0 \sqrt{2m\psi(\nu_{m(j-1)})}
\]
Note that by Gronwall's inequality, the second claim of Lemma \ref{lem:phi} indicates that
\[
\E_0 \phi(\nu_{mj})\nu_{mj}^s\leq \gamma_\phi^j\phi(\nu_0)\nu_{0}^s+\frac{1+a_r}{1-\gamma_\phi}.
\]
The bound for $\E \psi(\nu_{nm})$ comes from this, as a sum of $s=m$ and $s=m+1$. Combining both estimates of  Lemma \ref{lem:phi} using Cauchy Schwarz, we find that 
\[
\E_0 \sqrt{\psi(\nu_{m(j-1)})}
\leq \E_0 \sqrt{\E_{m(j-2)} \psi(\nu_{mj})}\leq \sqrt{\E_0(1+2a_r)\phi(\nu_{m(j-2)})}
\leq \sqrt{\gamma_\phi^j\phi(\nu_0)+\frac{1+a_r}{1-\gamma_\phi}}. 
\]
In summary we have our claims for $k=0$.

For nonzero $k<m$, by Markov property, the previous results indicates that 
\begin{equation}
\label{tmp:k1}
\E_k \sqrt{\psi(\nu_{nm+k})} \|e_{nm+k}\|_{C_{nm+k}^\rho}
\leq r^{-\frac{mn}{6}}\sqrt{\psi(\nu_k)}\|e_k\|_{C_k^\rho}+\sum_{j=1}^n r^{-\frac{m(n-j)}{6}}  \sqrt{\gamma_\phi^j\phi(\nu_k)+\frac{1+a_r}{1-\gamma_\phi}}
\end{equation}
\begin{equation}
\label{tmp:k2}
\E_k \psi(\nu_{mn+k})\leq \gamma_\phi^n\psi(\nu_k)+\frac{(1+a_r)^2}{1-\gamma_\phi}.
\end{equation}
Then by Cauchy Schwarz, Lemma \ref{lem:psikk} and that $\psi\geq \phi\nu^m\geq 1$
\begin{align*}
&(\E_0\sqrt{\psi(\nu_k)} \|e_k\|_{C_k^\rho})^2
\leq \E_0 \psi(\nu_k)\E_0 \|e_k\|^2_{C^\rho_k}\\
&\leq (1+\psi^4(\nu_0))\left((1+(1+\log\nu_0))\nu_0^k\phi(\nu_0)\|e_0\|^2_{C^\rho_0}
+2k(1+(1+\log\nu_0))\nu_0^k\phi(\nu_0)\right)\\
&\leq (1+\psi^4(\nu_0))\phi(\nu_0)(1+\nu_0)\nu_0^k(\|e_0\|^2_{C^\rho_0}+2k)
\leq 2(1+\psi^4(\nu_0))\psi(\nu_0)(\|e_0\|^2_{C^\rho_0}+2k). 
\end{align*}
Likewise by Lemma \ref{lem:psikk} we have $\E_0\psi(\nu_k)\leq 1+\psi^4(\nu_0)$, and
\[
\E_0\sqrt{\gamma_\phi^j\phi(\nu_k)+\frac{1+a_r}{1-\gamma_\phi}}
\leq \sqrt{\gamma_\phi^j(1+ \psi^4(\nu_0))+\frac{1+a_r}{1-\gamma_\phi}}.
\]
Plug these bounds into \eqref{tmp:k1} and \eqref{tmp:k2}, we have our claim. 
\end{proof}

\begin{proof}[Proof of Corollary \ref{cor:expstable}]
Let $U_{n,n_0}=\prod_{k=n_0}^{n-1}(I-\Khat_{k+1}H_k)A_k$. By iterating \eqref{tmp:pin} $n$ times, we find that
\[
\|C_n^\rho\|^{-1}U^T_{n,0}U_{n,0}\preceq U_{n,0}^T[C_n^\rho]^{-1}U_{n,0}\preceq\left( \prod_{k=1}^n\nu_k\mu_k/r\right)[C_{0}^\rho]^{-1}. 
\]
Taking spectral norm on both hand side yields $\|U_{n,0}\|\leq D_R\|[C_{0}^\rho]^{-1}\|\nu_n \prod_{k=1}^n\nu_k\mu_k/r$.

Recall that 
\[
\E_{nm} \sqrt{\psi(\nu_{(n+1)m})\prod_{k=nm+1}^{(n+1)m} \nu_k\mu_k/r}\leq \sqrt{(1+a_r )\phi(\nu_{nm})\nu^m_{nm}(r^{-\frac{m}{2}}+b_r\nu_{nm})}\leq r^{-\frac{m}{6}} \sqrt{\psi(\nu_{nm})}
\]
Therefore by iterative conditioning
\[
\E \sqrt{\psi(\nu_{(n+1)m+k})\prod_{j=1}^{(n+1)m+k} \nu_j\mu_j/r}
\leq r^{-\frac{m}{6}}\E\sqrt{\psi(\nu_{nm+k})\prod_{j=1}^{nm+k} \nu_j\mu_j/r}\leq \cdots\leq r^{-\frac{m(n+1)}{6}}\E \sqrt{\psi(\nu_k)}. 
\]
Apply Lemma \ref{lem:psikk}, $\E\sqrt{\psi(\nu_j)}$ can be bounded. Then notice that $\psi\geq 1$ yields our claim.
\end{proof}

\section{Concentration of noncentral random matrix}
\label{sec:RMT}
Random matrix theory (RMT) is one of the fastest growing branches of probability theory in recent years. Yet, most of the RMT results concern of the spectrum of a matrix $X$ where each column is i.i.d. with mean zero. For our application, each column of the  forecast ensemble spread matrix has its own non-central distribution, and we concern more of its ratio with respect to its expectation. It is of this reason, we have to develop the following result for EnKF. Fortunately, classic RMT arguments like $\epsilon$-net and Gaussian concentration are still valid for our purpose.
\par
The EnKF augmentations are crucial for such a result to hold. The additive inflation makes sure the matrix inversion is non-singular. The multiplicative inflation creates an important room so the concentration can fit in with high probability. The spectral projection makes sure the rank of the matrices are at most $p$.
\begin{thm}
\label{thm:RMT}
Let $\{a_k\}_{k\leq K}$ be $K$ vectors in $\reals^d$, and $\{\xi_k\}_{k\leq K}$ be $K$ i.i.d. $\mathcal{N}(0, \Sigma)$ random vectors in $\reals^d$. Denote $\Delta \xi_k=\xi_k-K^{-1}\sum_{j=1}^K \xi_j$, and 
\[
C=\frac{1}{K-1}\sum_{k=1}^K a_k\otimes a_k,\quad D=C+\Sigma, \quad Z=\frac{1}{K-1}\sum_{k=1}^K (a_k+\Delta \xi_k)\otimes (a_k+\Delta\xi_k).
\]
Clearly  $\E Z=D$. Fixed any $\rho>0$, denote the condition number of $C+\rho I_d$ as $\mathcal{C}$, and 
\[
\mu=\inf\{r\geq 0: [Z+\rho I_d]^{-1}\preceq r[D+\rho I_d]^{-1} \},\quad
\lambda=\inf\{r\geq 0: Z\preceq r [D+\rho I_d]\}, 
\]
Suppose that rank$([a_1,\ldots, a_K])\leq p$, rank$(\Sigma)\leq p$. For any fixed $\delta>0$, there are constants $D_\delta $ and $c_\delta $ such that when $K>D_\delta/c_\delta p$, 
\begin{itemize}
\item With high probability, both $\lambda$ and  $\mu$ are close to $1$: for the event $\mathcal{U}:=\{\lambda>1+5\delta\text{ or } \mu>1+5\delta \}$, 
\[
\Prob(\mathcal{U})\leq  (\log \mathcal{C}+1)  \exp(D_\delta  p-c_\delta  K).
\]
\item In the complementary set, $\mu$ is controllable through its moment. Suppose that $\rho\leq \|\Sigma\|$, then for any fixed number $n$, there is a constant $C_n$
\[
\E \unit_{\mathcal{U}}\mu^n\leq C_n \rho^{-n}\|\Sigma\|^n (\log \mathcal{C}+1) \exp(D_\delta  p-c_\delta K). 
\]
And $\lambda$ is controllable through its tail: when $K\geq D_\delta  p$, so any $t>0$
\[
\Prob (\lambda>8+t)\leq \exp(-c_\delta K t). 
\]
\end{itemize}
\end{thm}

\begin{proof}
\textbf{Step 1: spectral projection.} 
In our discussion below, without lost of generality, we assume $\delta$ is a small positive number, so the following simplified estimation holds
\[
(1+b\delta)(1+a\delta)\leq \frac{1+b\delta}{1-a \delta}\leq 1+\min\{a+b+1, 2(a+b)\}\delta
\]
 for all $a,b\in [0,4]$.  Now let 
\[
F_v:=\frac{v^T(D+\rho I_d) v}{v^T(Z+\rho I_d)v},\quad G_v:=\frac{vZv^T}{v(D+\rho I_d)v^T}. 
\]
Then clearly
\[
\mu=\sup_{v\in \reals^d}F_v,\quad \lambda=\sup_{v\in \reals^d}G_v. 
\]
\par
 Next, let $\mathcal{P}$ be the linear sum of the column space of $\Sigma$ and the linear space spanned  by $\{a_1,\ldots, a_K\}$. By our condition, $\mathcal{P}$ has dimension at most $2p$. Denote the projection of $v$ to $\mathcal{P}$ as $v'$, and the residual as $v_\bot$. Note that $\langle a_k, v_\bot\rangle=0$ and $\langle \Delta\xi_k, v_\bot\rangle=0$ a.s., 
 \[
 F_v=\frac{\frac{1}{K-1}\sum \langle a_k, v'\rangle^2 + v'^{T} \Sigma v'+\rho |v'|^2+\rho|v_\bot|^2}{
 \frac{1}{K-1}\sum \langle a_k+\Delta \xi_k, v'\rangle^2 + \rho |v'|^2+\rho |v_\bot|^2}
 \quad a.s.,
 \]
 \[
 G_v=\frac{\frac{1}{K-1}\sum \langle a_k+\Delta \xi_k, v'\rangle^2 }{\frac{1}{K-1}\sum \langle a_k, v'\rangle^2 + v^{'T} \Sigma v'+\rho |v'|^2+\rho |v_\bot|^2}\geq G_{v'}\quad a.s.. 
 \]
One elementary fact is that if  $a,b,c$ are nonnegative real numbers. 
\begin{equation}
\label{eqn:quotient}
\frac{a+b}{a+c}\leq \max\left\{1,\frac{b}{c}\right\}.
\end{equation}
 As a consequence $F_v\leq \max\{1, F_{v'}\}$. Since we do not concern about the part of $\mu$ that is below $1$, we can focus on $v\in \mathcal{P}$. Moreover, because $F_v$ is invariant under renormalization, we can focus on $|v|=1$.\\
\noindent \textbf{Step 2: two $\epsilon$-nets.} In order to parameterize  $v\in \mathcal{P}$, let $\Psi\Lambda \Psi^T$ be the orthogonal decomposition of $\Sigma$, where $\Psi$ is a $d \times p $ matrix, and $\Lambda$ is $p\times p$ with strictly positive diagonal entries.  Let $\Theta= \Lambda^{-1/2}\Psi^T$. The set $\{\Theta^T u, u\in \mathcal{S}^{p-1}\}=\{v\in \mathcal{P},|v|=1\}$, where $\mathcal{S}^{p-1}=\{u\in \reals^p, |u|=1\}$ is the $p-1$ dimensional sphere. With a transformation  through $\Theta$, we denote the spread matrices consist of the dynamical forecast and the system noise as follows
\[
S=[\Theta a_1, \ldots, \Theta a_K]^T,\quad S'=(S^TS+(K-1)\rho\Theta\Theta^T)^{1/2}, 
\]
and also the random matrices
\[
T=[\Theta \xi_1,\ldots, \Theta \xi_K]^T,\quad T'=[\Theta \Delta\xi_1,\ldots, \Theta \Delta\xi_K]^T=Q_KT,\quad Q_K=I_K-K^{-2}\vec{1}_K\otimes \vec{1}_K. 
\]
$\vec{1}_K$ is the $K$ dimensional vector with all entries being $1$. Note that $\|Q_K\|\leq 1$.  One important fact will be exploited in later estimation is that $T$ is a $K\times p$ random matrix with i.i.d. $\mathcal{N}(0, 1)$ entries. Also note with these notations, 
\[
F_{\Theta^Tu}=\frac{|S' u|^2 +(K-1)}{
\sum_{k=1}^K( [S u]_k+[T' u]_k)^2  + (K-1)\rho |\Theta^T u|^2}, \quad G_{\Theta^T u}=\frac{ |S u+T'u|^2}{|S' u|^2+(K-1)},
\]
where $[x]_k$ denotes the $k$-th coordinate of a vector $x$. Then $F_{\Theta^Tu}= f_{u,u}, G_{\Theta^T u}= g_{u,u} $, where
\[
f_{u,w}:=\frac{|S' u |^2 +(K-1)}{
 \sum_{k=1}^K  ([S u]_k +[T' w]_k)^2+(K-1)\rho |\Theta^T u|^2},\quad
g_{u,w}:=\frac{ |S u +T' w|^2}{|S' u |^2 +(K-1)}.
\]
Like many other random matrix problems, $f_{u,u}, g_{u,u}$ are easy to estimate for a fixed $u$, but difficult to estimate over all $u\in \mathcal{S}^{p-1}$. The general solution to such problem is finding proper $\epsilon$-nets, where $\epsilon$ is a very small positive number to be fixed later. Here we need two. 
\par
Pick a group points $\{w_j\in \reals^p| j\in J,  |w_j|=1\}$ as an $\epsilon$-net for the $p-1$ dimensional sphere $\mathcal{S}^{p-1}$. In other words, 
$\mathcal{S}^{p-1}\subset\cup_j B_{\epsilon}(w_j)$, where $B_r (x)$ is a ball of radius $r$ around point $x$.  By Lemma 5.2 \cite{Ver11} , the cardinality of this net is bounded by $|J|\leq (1+\frac{2}{\epsilon})^p$.
\par
As we are dealing with another matrix $S'$, we need a second net $\{u_i\in \reals^p|i\in I\}$ generated by the norm $|S' \cdot |$, so
$\mathcal{S}^{p-1}\subset\cup_i B^{S}_{\epsilon}(u_i)$. Here 
\[
u\in B^{S}_{\epsilon}(u_i)\quad \text{if} \quad |S' (u-u_i)|\leq \epsilon |S' u_i|.
\] 
By Lemma \ref{lem:epsilonnet}, this set has cardinality 
\[
|I|\leq 4(1+4\epsilon^{-1})^pp(\log\mathcal{C}'+1),
\] with $\mathcal{C}'$ being the condition number of $S'$. 
\par
Note that 
\[
F(\xi_1,\ldots,\xi_K)\leq \sup_{u\in \mathcal{S}^{p-1}} \max\{F_{\Theta^Tu},1\}\leq \max_{i\in I,j\in J}\left\{ \sup_{u\in B^{S}_\epsilon (u_i)} \sup_{w\in B_{\epsilon}(w_j)} f_{u,w}, 1\right\}.
\]
To continue, we will find simpler bounds for $F_{\Theta^T u}, G_{\Theta^T u}$ when $u\in B^{S} (u_i), u\in B_{\epsilon}(w_j)$. This discussion will be done separately for two different scenarios.

\noindent\textbf{Step 3: $|S' u_i|\leq \frac{1}{2}\delta\sqrt{ K-1}$}. In this case, the nominator of $f_{u,w}$ is bounded by 
\[
(1+\epsilon)^2|S' u_i|^2+(K-1)\leq (K-1)(1+\tfrac{1}{2}\delta^2). 
\]
and denominator of $g_{u,u}$ bounded from below by $K-1$. As for the denominator of $f_{u,u}$, and nominator of $g_{u,u}$,   $|T' w_j|^2$ makes a good approximation. In specific, because   $|T'(w_j-w)|\leq \epsilon \|T'\|$ and 
\[ 
|Su|\leq |S'u|\leq |S'(u-u_i)|+|S' u|\leq \frac{1+\epsilon}{2}\delta\sqrt{K-1}.
\]
By Cauchy Schwarz , 
\begin{align*}
\sum_{k=1}^K  ([S u]_k +[T' w]_k)^2 &=\sum_{k=1}^K  ([T' w_j]_k+[T'(w-w_j)]_k+[S u]_k)^2\\
&\geq  \sum_{k=1}^K [T' w_j]_k^2-2[T' w_j]_k([T'(w-w_j)]_k+[S u]_k)\\
&\geq  |T' w_j|^2-2\sqrt{\sum_k|T' w_j|^2_k}\sqrt{2\sum_{k} (|T'(w-w_j)|^2_k+[S u]_k^2)}\\
&\geq |T' w_j|^2-4|T' w_j|(\epsilon \|T'\|+\frac{1+\epsilon}{2}\delta\sqrt{K-1}).
\end{align*}
Likewise, the nominator of $g_{u,w}$ is bounded by 
\begin{align*}
|Su+T'w|^2 &\leq  \left(|T' w_j|+|T'(w-w_j)|+|S u|\right)^2\\
&\leq (|T' w_j|+\epsilon \|T'\|+\tfrac{1+\epsilon}{2}\delta\sqrt{K-1})^2.
\end{align*}
Recall that $T w_j\sim \mathcal{N}(0, I_k)$, so $|T' w_j|=|Q_K T w_j|\approx \sqrt{K-1}$ by concentrations of Gaussian variables, then the quantity above can be bounded. In particular, let us consider the events of large deviations: 
\[
\mathcal{D}=\{\omega: \|T'\|\geq \tfrac{\delta(1-\epsilon)}{2\epsilon}\sqrt{ (K-1)} \},\quad \mathcal{A}_{i,j}=\left\{\omega: |T' w_j|\leq \sqrt{\frac{K-1}{1+\frac{1}{2}\delta} } \text{ or }|T' w_j|\geq(1+ \tfrac{\delta}{2})\sqrt{K-1}\right\}.
\]
Then in the canonical event $(\mathcal{D}\cup \mathcal{A}_{i,j})^c$, 
\[
(1-3\delta)(K-1)\leq \sqrt{\frac{K-1}{1+\frac{1}{2}\delta}}(\sqrt{\frac{K-1}{1+\frac{1}{2}\delta}}-2\delta\sqrt{K-1})\leq \sum_{k=1}^K  ([S u]_k +[T' w]_k)^2\leq (1+3\delta) (K-1). 
\]
In this canonical set, we have a good upper bound, 
\[
F_{\Theta^T u}= f_{u,u}\leq \frac{1+\frac{1}{2}\delta^2}{1-3\delta}\leq 1+5\delta,\quad 
G_{\Theta^T u}\leq \frac{(K-1)(1+3\delta)}{K-1}=1+3\delta. 
\]
On the other hand, because $\|T'\|\leq \|T\|$, so by spectral norm estimate of the Gaussian random matrix $T$, Corollary 5.35 \cite{Ver11},
\[
\Prob(\mathcal{D})\leq p_D:=\exp(-(\tfrac{\delta(1-\epsilon)}{2\epsilon}\sqrt{K-1}-\sqrt{K}-\sqrt{p})^2);
\]
moreover, because $\E |Q_K T w_j|^2=K-1$,  by Hansen-Wright's inequality for Gaussian variables \cite{RV13}, for some constants $D_\delta , c_\delta >0$,
\[
\Prob(\mathcal{A}_{i,j})\leq D_\delta \exp(-c_\delta  K). 
\]
A different pair of $c_\delta , D_\delta $ will make
\[
\Prob(\mathcal{D})+\Prob(\mathcal{A}_{i,j})\leq D_\delta  \exp(-c_\delta  K). 
\]
While in the non canonical set $\mathcal{D}\cup \mathcal{A}_{i,j}$, we can use a trivial upper bound for $F_{\Theta^T u}$
\[
F_{\Theta^T u}\leq \frac{|S' u|^2+(K-1)}{\rho |\Theta^T u|^2}\leq U_i:=\rho^{-1}\|\Sigma\|(1+\tfrac{1}{2}\delta)(K-1). 
\]
The part for $G_{\Theta^T u}$ can be achieved through a Cauchy inequality in the end. 
\par
\noindent\textbf{Step 4: $|S' u_i|\geq \tfrac{1}{2}\delta\sqrt{K-1}$}. In this case, the nominator of $f_{u,w}$ is bounded 
\begin{equation}
\label{tmp:F1}
(1+\epsilon)^2|S' u_i|^2+(K-1)\leq (1+\tfrac{1}{8}\delta)(|S' u_i|^2+(K-1)),
\end{equation}
and the denominator of $g_{u,w}$ is bounded from below by 
\begin{equation}
\label{tmp:G1}
(1-\epsilon)^2|S'u_i|+(K-1)\geq (1-\tfrac{1}{8}\delta)(|S'u_i|^2+(K-1)). 
\end{equation}
 Next we try to bound the denominator of $f_{u,w}$ and the nominator of $g_{u,w}$. Denote
\[
F_{i,j}=\sqrt{\sum_k ([S u_i]_k+[T' w_j]_k)^2+(K-1)\rho |\Theta^Tu_i|^2},\quad G_{i,j}=|Su_i+T'w_j|. 
\]
 Then the fact that $|S' (u-u_i)|\leq \epsilon |S' u_i|$, and that $|w-w_j|\leq \epsilon$,  the denominator of $f_{u,w}$ is bounded from below by:
\begin{align}
\notag
&\sum_k  ([S u]_k +[T' w]_k)^2+\rho (K-1) |\Theta^Tu|^2\\
\notag
&\geq \sum_k  ([S u_i]_k +[T' w_j]_k)^2-2([S (u_i-u)]_k+[T' (w_j-w)]_k)([S u_i]_k +[T' w_j]_k)\\
\notag
&\quad+\rho (K-1) |\Theta^Tu_i|^2-2\rho (K-1)|\Theta^T (u-u_i)||\Theta^T u_i|\\
\notag
&\geq F_{i,j}^2-2\sqrt{\sum_k ([S u_i]_k +[T' w_j]_k)^2+\rho (K-1)|\Theta^T u_i|^2}\times\\
\notag
&\qquad\sqrt{ 2\sum_k ([S (u-u_i)]^2_k + [T'(w- w_j)]^2_k)+2\rho (K-1) |\Theta^T(u-u_i)|^2}\\
\label{tmp:F2}
&\geq F^2_{i,j}-2F_{i,j}\sqrt{2|S' (u-u_i)|^2+ 2|T' (w-w_j)|^2} \geq F^2_{i,j}-2\epsilon F_{i,j} \sqrt{2|S' u_i|^2 +2\|T' \|^2 }.
\end{align}
Likewise,  the nominator of $g_{u,w}$ is bounded by:
\begin{align}
\notag
|Su+T'w|^2&\leq (|Su_i+T'w_j|+|S(u-u_i)|+|T'(w-w_j)|)^2 \\
\label{tmp:G2}
&\leq (G_{i,j}+\epsilon (|S'u_i|+|T'|))^2. 
\end{align}
In order to continue, we rewrite $F_{i,j}, G_{i,j}$ by rotating proper terms. There is a $K\times K$ rotation matrix $\Phi_i$, with its first row being 
\[
([S u_i]_1/|S u_i|,[S u_i]_2/|S u_i|,\ldots, [S u_i]_K/|S u_i|).
\]
 Then if we consider the following $K\times 1$ random vector 
\[
\zeta=(\zeta_1,\ldots,\zeta_K)^T= \Phi_i T' w_j =\Phi_i Q_K Tw_j
\]
we have 
\[
|S u_i|\zeta_1=-\sum_{k=1}^K [S u_i]_k [T'w_j]_k, \quad \sum^K_{k=1} |\zeta_k|^2=\sum_{k=1}^K  [T'w_j]^2_k. 
\]
Moreover, because $Tw_j\sim \mathcal{N}(0, I_K)$, $\zeta\sim \mathcal{N}(0,\Phi_i Q_K^2 \Phi_i^T)$.  In particular,
\begin{align}
\notag
F_{i,j}^2&=\sum_{k=1}^K ([S u_i]_k+[T' w_j]_k)^2+\rho|\Theta^T u_i|^2\\
\notag
&= (|S u_i|-\zeta_1)^2+\rho|\Theta^T u_i|^2+\sum_{k=2}^{K} |\zeta_k|^2\\
\label{tmp:F3}
&\geq (|S' u_i|-|\zeta_1|)^2+\sum_{k=2}^{K} |\zeta_k|^2. 
\end{align}
And 
\[
G_{i,j}^2=|Su_i+T'w_j|^2=(|S u_i|-\zeta_1)^2+ \sum_{k=2}^{K} |\zeta_k|^2\leq (|S' u_i|+|\zeta_1|)^2+ \sum_{k=2}^{K} |\zeta_k|^2. 
\]
We consider these sets of large deviations, 
\[
\mathcal{A}_{i,j}=\left\{|\zeta_1|\geq \frac{\delta}{8}|S' u_i| \right\},\quad \mathcal{B}_{i,j}=\left\{\left|\sum_{k=2}^{K} |\zeta_k|^2-(K-1)\right|\geq \frac{1}{2}\delta(K-1)+\frac{\delta}{4}|S' u_i|^2 \right\},
\]
and $\mathcal{D}_i=\{\|T'\|\geq (4\epsilon)^{-1}\delta\sqrt{ (K-1)+|S' u_i|^2} \}$. Then in the canonical set $(\mathcal{D}_i\cup \mathcal{A}_{i,j}\cup \mathcal{B}_{i,j})^c$,
\[
F^2_{i,j}\geq \eqref{tmp:F3}\geq (1-\frac{1}{2}\delta)|S' u_i|^2+(1-\frac{1}{2}\delta)(K-1),
\]
and for small enough $\delta$, 
\[
 \eqref{tmp:F2}\geq  F_{i,j}(F_{i,j}-2\epsilon\sqrt{2|S' u_i|^2+2\|T'\|^2})\geq (1-\tfrac{3}{4}\delta)(|S' u_i|^2+K-1).
\]
Combine this with \eqref{tmp:F1},  $f_{u,u}\leq \frac{1+\frac{1}{8}\delta}{1-\frac{3}{4}\delta}\leq 1+2\delta$. 

Likewise, in the canonical set $(\mathcal{D}_i\cup \mathcal{A}_{i,j}\cup \mathcal{B}_{i,j})^c$, $G^2_{i,j}\leq (1+\tfrac{\delta}{2})^2 (|S'u_i|^2+(K-1)), $ and by \eqref{tmp:G1} and \eqref{tmp:G2}, 
\[
g_{u,u}\leq \frac{(1+\tfrac{\delta}{2})^2 (1+\delta)^2(|S' u_i|^2+(K-1))}{(1-\tfrac{1}{8}\delta)(|S'u_i|^2+(K-1))}\leq 1+5\delta. 
\]

Next we bound the probability of large deviations. By Gaussian tail estimate, there are constant $D_\delta,c_\delta>0$ such that  
\[
\Prob(\mathcal{A}_{i,j})\leq C\exp\left(-c|S' u_i|^2\right).
\] 
And because $\|T'\|\leq \|T\|$ and $|S'_n u_i|\geq \frac{1}{2}\sqrt{\delta(K-1)}$, by concentration of  spectral norm of random matrices, Corollary 5.35 \cite{Ver11}, 
there is a pair of constants $(D_\delta , c_\delta )$ such that 
\[
\Prob(\mathcal{D}_i)\leq D_\delta \exp(-c_\delta  (\sqrt{K}+\sqrt{p}-(2\epsilon)^{-1}\sqrt{\delta (K-1)+\delta|S' u_i|^2})^2).
\]
Lastly, we can compute the mean of $\sum_{k=2}^K \zeta_k^2$
\[
\E \sum_{k=2}^K \zeta_k^2=\E|\zeta|^2-\E \zeta_1^2=\text{tr}(\Psi_i Q_K^2\Psi_i^T)-e^T_1\Psi_i Q_K^2 \Psi_i^T e_1,
\]
where $e_1=[1,0,\cdots,0]^T$. Since $\text{tr}(\Psi_i Q_K^2\Psi_i^T)=\text{tr}(Q_K^2)=(K-1)^2/K$, and 
\[
e^T_1\Psi_i Q_K^2 \Psi_i^T e_1\leq \|Q_K^2\|\leq 1. 
\]
So $\E \sum_{k=2}^K \zeta_k^2\geq K-3$. Moreover, because $\zeta\sim \mathcal{N}(0,\Phi_i Q_K^2 \Phi_i^T)$, $\sum_{k=2}^K \zeta_k^2$ has the same distribution as $|U\xi|^2$, where $\xi\sim \mathcal{N}(0,I_K)$, and
\[
U= \Psi_i Q_K- E_{1,1}\Psi_i Q_K,\quad E_{1,1}=e_1\otimes e_1.
\]  
Then 
\[
\|U^TU\|=\|Q_K\Psi^T_i(I-E_{1,1})\Psi_i Q_K\|\leq 1=\|I_K\|.
\]
Also the Hilbert-Schmidt norm is bounded by 
\begin{align*}
\|U^TU\|_{HS}^2&=\text{tr}(Q_K\Psi^T_i(I-E_{1,1})\Psi_i Q^2_K\Psi_i^T(I-E_{1,1})\Psi_i Q_K)\\
&=\text{tr}((I-E_{1,1})\Psi_i Q^2_K\Psi_i^T(I-E_{1,1})\Psi_i Q^2_K\Psi^T_i)\\
&\leq \text{tr}(\Psi_i Q^2_K\Psi_i^T(I-E_{1,1})\Psi_i Q^2_K\Psi^T_i)\\
&= \text{tr}(\Psi_i^TQ_K^4 \Psi_i(I-E_{1,1}))\leq tr(Q_K^4)\leq K=\|I_K\|^2_{HS},
\end{align*}
where we used that $\text{tr}(AB)=\text{tr}(BA)$ and $\text{tr}((I-E_{1,1}) A)\leq \text{tr}(A)$ for all PSD $A$. So by the Hansen-Wright's inequality, \cite{RV13}, there are constants $D_\delta $ and $c_\delta $
\[
\Prob(\mathcal{B}_{i,j})\leq D_\delta \exp(-c_\delta  ((K-1)+|S' u_i|^2)).
\]
By the union bound, we can choose a different pair of $D_\delta $ and $c_\delta $
\[
\Prob(\mathcal{D}_i\cup \mathcal{A}_{i,j}\cup \mathcal{B}_{i,j})\leq D_\delta \exp(-c_\delta  |S' u_i|^2-c_\delta K). 
\]
Moreover, when in the rare event $\mathcal{D}_i\cup \mathcal{A}_{i,j}\cup \mathcal{B}_{i,j}$, we can again use the trivial upper bound
\[
F_{\Theta^T u}\leq \frac{(1+\epsilon)^2|S' u_i|^2+(K-1)}{\rho |\Theta^T u|^2}\leq U_i:=\rho^{-1}(1+\tfrac{1}{2}\delta)\|\Sigma\|((K-1)+|S'u_i|^2). 
\]

\noindent\textbf{Step 5. Summing up.} Finally, we can put all our estimates together. First, denote the union of all large deviation set as 
\[
\mathcal{U}=\mathcal{D}\cup \bigcup_{i,j} (\mathcal{D}_i\cup \mathcal{A}_{i,j} \cup  \mathcal{B}_{i,j}).
\]
Based on previous discussion, outside $\mathcal{U}$, $F_{\Theta^T u},G_{\Theta^T u}\leq 1+5\delta$ for all $u$, so $\lambda, \mu\leq 1+5\delta$. While the probability of this event $\mathcal{U}$ is bounded by the union bound as
\[
\Prob(\mathcal{U})\leq p_{D}+|I||J|D_\delta \exp(-c_\delta  K)\leq  (\log \mathcal{C}'+1) \exp(\log p+D_\delta  p-c_\delta  K).
\]
Finally notice that the condition number $C'$ of matrix $S'$ is dominated by the one of $C$ by Lemma \ref{lem:cond} since $\Theta$ has rank $p$:
\[
\mathcal{C}'^{2}=\text{Cond}(S^T S+\rho(K-1) \Theta \Theta^T)=\text{Cond}(\Theta (C+\rho I)\Theta^T)\leq \mathcal{C}^2. 
\]

For the control of $\mu$ in the rare event, recall the trivial upper bounds. When $|S'u_i|\leq \frac{1}{2}\sqrt{\delta(K-1)}$, 
\begin{equation}
\label{tmp:1case}
\Prob(\mathcal{D}\cup \mathcal{A}_{i,j})U^n_i\leq D_\delta \rho^{-n}\|\Sigma\|^n(1+\tfrac{1}{2}\delta)^n(K-1)^n\exp(-c_\delta  K).
\end{equation}
For $|S'u_i|\geq \frac{1}{2}\sqrt{\delta(K-1)}$, 
\begin{equation}
\label{tmp:2case}
\Prob(\mathcal{D}_i\cup \mathcal{A}_{i,j}\cup\mathcal{B}_{i,j})U^n_i\leq D_\delta \rho^{-n}\|\Sigma\|^n(1+\tfrac{1}{2}\delta)^n((K-1)+|S'u_i|^2)^n\exp(-c_\delta  K-c_\delta |S'u_i|^2).
\end{equation}
By maximizing the right hand side  of \eqref{tmp:2case} over all possible value of  $|S'u_i|$, we can find a new pair of $c_\delta , D_\delta $ such that 
\[
\eqref{tmp:1case}\leq \rho^{-n}\|\Sigma\|^nD_\delta \exp(-c_\delta  K),\quad \eqref{tmp:2case}\leq \rho^{-n}\|\Sigma\|^n D_\delta \exp(-c_\delta K).
\]
As a consequence,
\begin{align*}
\E \unit_{\mathcal{U}} \mu^n&\leq \sum_{(i,j):|S'u_i|\leq \frac{1}{2}\sqrt{\delta(K-1)}}\Prob(\mathcal{D}\cup \mathcal{A}_{i,j})U^n_i+
\sum_{(i,j):|S'u_i|> \frac{1}{2}\sqrt{\delta(K-1)}}\Prob(\mathcal{D}_i\cup \mathcal{A}_{i,j}\cup\mathcal{B}_{i,j})U^n_i\\
&\leq |I||J|\rho^{-n}\|\Sigma\|^nD_\delta  \exp(-c_\delta  K)\\
&\leq  (\log \mathcal{C} +1) \rho^{-n}\|\Sigma\|^n\exp(\log p+\log D_\delta+D_\delta p-c_\delta K),
\end{align*}
for a different pair of $c_\delta , D_\delta >0$. Moreover, we can remove the $\log p+\log D_\delta$ term by having a larger $D_\delta$ in front of $p$.

The tail of $\lambda$ can be achieved by the following trivial bound
For the $\lambda^m$ part, notice that  by \eqref{eqn:quotient}
\[
g_{u,u}\leq \frac{2|S'u|^2+2|T'u|^2}{|S'u|^2+(K-1)}\leq 2\max\left\{\frac{\|T'\|^2}{K-1}, 1\right\}. 
\]
So by Corollary 3.53 of \cite{Ver11},  $\Prob(\|T'\|>\sqrt{K-1}+\sqrt{p}+t)\leq \exp(-t^2/2)$, because $p<K-1$, with a proper $c_\delta $ 
\[
\Prob(\lambda\geq 8+t)\leq \exp(-c_\delta K t). 
\]
\end{proof}

\begin{lem}
\label{lem:epsilonnet}
Let $S$ be any $p\times p$ nonsingular matrix, and $\mathcal{C}(S)$ be its conditional number. Then for any fixed $\epsilon\in (0,1)$, we say $U\subset\mathcal{S}^{p-1}$ is a $S$-relative $\epsilon$-net  of $\mathcal{S}^{p-1}$, if for any $x\in \mathcal{S}^{p-1}$, there is a $u\in U$ such that $|S(x-u)|\leq \epsilon |S u|.$ Then there exists a $S$-relative $\epsilon$-net  of $\mathcal{S}^{p-1}$,
\[
|U|\leq e(1+2e^{\frac{1}{p}}\epsilon^{-1})^pp(\log \mathcal{C}(S)+1),
\]
which can be further bounded by $4(1+4\epsilon^{-1})^p(p\log \mathcal{C}(S)+1)$ for simplicity.  Moreover, the linear dependence of $|U|$ over $\log \mathcal{C}(S)$ is sharp. 
\end{lem}
\begin{proof}
Denote $m$ to be the minimum singular value of $S$, and let $c>1$ be a number to be determined. For $n=1,\ldots, N_c:=\lceil \log (\mathcal{C}(S))/\log c\rceil$, let 
\[
D_n:=\{u: |u|=1, m c^{n-1} \leq |Su|\leq m c^{n} \}.
\]
Then $\mathcal{S}^{p-1}=\cup_{n=1}^{N_c} D_n$, and we will construct a $S$-relative $\epsilon$-net for each $D_n$. 
\par
We say $U\subset D_n$ is a $S$-relative $\epsilon$-separated set of $D_n$, if for two different $x,y\in U$, $|S(x-y)|\geq \epsilon\max\{|Sx|, |Sy|\}$. Then, the cardinality of any $S$-relative $\epsilon$-net of $D_n$ is upper bounded. To see this, note that if $x,y\in U$,  $B_{\frac{1}{2}\epsilon|S x|}(Sx)$ and $B_{\frac{1}{2}\epsilon|S y|}(Sy)$ have  no intersection: else  $|Sx-Sy|< \frac{\epsilon}{2} (|Sx|+|Sy|)$, which contradicts the definitions of $\epsilon$-separation. On the other hand, because $|Sx|\leq m c^n$, $B_{\frac{1}{2}\epsilon|S x|}(Sx)\subseteq B_{(1+\frac{1}{2}\epsilon)mc^n } (0)$, so 
\[
|U|\text{Vol}(B_{\frac{1}{2}\epsilon m c^{n-1}}(Sx))\leq \text{Vol}(B_{(1+\frac{1}{2}\epsilon)mc^n } (0))\quad \Rightarrow \quad |U|\leq (1+2\epsilon^{-1})^p c^p. 
\]
\par
Since $S$-relative $\epsilon$-separated sets of $D_n$ all have finite cardinality, there is a $U'_n$ among them with the maximal cardinality. We claim this $U'_n$ is a $S$-relative $c\epsilon$-net for $D_n$. To see this, suppose there is a $z$, with $|S(z-x)|\geq c\epsilon |Sx|$ for all $x\in U'_n$, then $|S(z-x)|\geq \epsilon|Sz|$ because both $z$ and $x$ are in $D_n$, and their S-norms are at most $c$ multiple apart. So $U'_n\cup\{z\}$ is another $S$-relative $\epsilon$-separated sets of $D_n$, which contradicts that $U'_n$ has the maximal cardinality. 
\par
To summarize, we can use $\cup_{n=1}^{N_c} U'_n$ as a $c\epsilon$ set for $\mathcal{S}^{p-1}$, with the total cardinality bounded by $(1+2\epsilon^{-1})^p c^pN_c$. To show our claim, we will replace $\epsilon$ with $c^{-1}\epsilon$ and $c=e^{\frac{1}{p}}$. 

It is also easy to see the linear dependence of $|U|$ on $\mathcal{C}(S)$. Let $F(x)=\log |Sx|$, then $F(\mathcal{S}^{p-1})=[\log m, \log m+\log\mathcal{C}(S)]$, where $m$ is the minimum eigenvalue of $S$. For each $u\in U$, if $|S(x-u)|\leq \epsilon |Su| $, then $F(x)\in [\log(1-\epsilon)+\log F(u), \log(1+\epsilon)+\log F(u)]$, which is an interval of fixed length $\log \frac{1+\epsilon}{1-\epsilon}$. Therefore 
\[
|U|\geq \frac{\log\mathcal{C}(S)}{\log \frac{1+\epsilon}{1-\epsilon}}. 
\]
\end{proof} 

\section{Conclusion and discussion}
\label{sec:conclude}
Ensemble Kalman filters (EnKF) are indispensable data assimilation methods for high dimensional systems. Their surprisingly successful performance with ensemble size $K$ much lower than ambient space dimension $d$ has remained a mystery for mathematicians. The practitioners often attribute this success to the existence of a  low  effective dimension $p$, of which the formal definition is sorely lacking. This paper closes this gap by considering a Kalman filter with random coefficients, and uses its covariance $\Rtilde_n$ as an intrinsic filtering performance criteria. The effective dimension then can naturally be defined as the number of eigenvalues of $\Rtilde_n$, or a system instability matrix, that are above a significance threshold $\rho$. An EnKF with proper covariance inflation and localization techniques is constructed by Algorithms \ref{alg:EnKF}, which exploits the low effective dimension structure. Then assuming the system is uniformly observable, and the ensemble size exceeds a constant multiple of $p$, Theorem \ref{thm:weak} asserts that the Mahalanobis error of EnKF decays geometrically to a constant. Useful properties such as corvariance fidelity, exponential stability, and filter accuracy of EnKF  come  as immediate corollaries.  This framework can be directly applied to a simple stochastic turbulence, where the effective filtering dimension along with the EnKF tuning parameters are explicitly computable. The proof exploits the intrinsic Mahalanobis error dissipation,  shows this mechanism operates with high probability using a new noncentral random matrix concentration result, and regulates the behavior of outlier by designing a Lyapunov function. 
\par
As the first step of studying EnKF performance, this paper looks at a relatively simple setup. There are multiple directions this framework can be improved at. Here we list five of them to inspire for future research:
\begin{itemize}
\item The random linear coefficient setting \eqref{sys:random} includes a wide range of applications \cite{Bou93, MT16uq}. But many geophysical applications of EnKF, the forecast models are fully nonlinear \cite{TMK15non, TMK15}. Two major challenges arise if one wants to generalize our setup to nonlinear systems. First it will be difficult for the forecast ensemble to capture the forecast dynamics, so the error dynamics will no longer be linear. Second, the reference Kalman filter plays a pivotal role in this framework, it is not clear what would replace it in a nonlinear setting.
\item One important feature that this paper has not discussed is the universality of noise. The authors believe the same proof remains valid if the noise distribution are assumed only to be sub-Gaussian \cite{Ver11}. This is because the only feature we have used about Gaussian distributions is their concentration, which also holds for sub-Gaussian distributions or even exponential distributions. The non-Gaussian noise may be an essential ingredient if one wants to investigate the nonlinear settings. This paper has not discussed this universality intentionally, as it is less known for the filtering community, and may cause confusion.
\item The Mahalanobis error appears to be a natural statistic for covariance fidelity and stability analysis. It is equivalent to the average $l_2$ error $\E |e_n|$, but this equivalence is not good in a high dimensional setting. For example, in Theorem \ref{thm:weak}, $\E|e_n|$ has a square root dependence on $d$. In practice one would expect the dependence should be $\sqrt{p}$. To reach this result, one would need to investigate the uncertainty levels in different dimensions. 
\item The uniform observability Assumption \ref{aspt:lowdim} is proposed for a simpler Lyapunov function construction. It is stronger than the general assumptions developed in the classic Kalman filter literature such as \cite{Bou93}. It will be interesting if our framework can be further generalized in this aspect.
\item In recent years, many other EnKF augmentations are invented based on various intuition. Whether our framework can justify their formulations require further investigation. 
\end{itemize}

\section*{Acknowledgement}
This research is supported by the MURI award grant N00014-16-1-2161, where A.J.M. is the principal investigator, while X.T.T. was supported as a postdoctoral fellow. This research is also supported by the NUS grant R-146-000-226-133, where X.T.T. is the principal investigator. The authors thank Ramon van Handel for his discussion on Lemma \ref{lem:epsilonnet},  Kim Chuan Toh for his discussion on Lemma \ref{lem:Kalconcave}, and the anonymous referee for the presentation suggestions. 
\appendix
\section*{Appendix}
\section{Matrix inequalities}
The following lemma has been mentioned in \cite{FB07} for dimension one. 
\begin{lem}
\label{lem:Kalconcave}
The prior-posterior Kalman covariance update mapping $\Kalman_{n}$ in \eqref{sys:optimal}, can also be defined as
\[
\Kalman_{n}(C)=(I-G H_n) C(I-G H_n)^T + G  G^T
\]
where $G:=C H_n (I +H_n C H_n^T)^{-1}$ is the corresponding Kalman gain. $\Kalman_{n}$ is a  concave monotone operator from PD to itself. 
\end{lem}
\begin{proof} The first matrix identity is straightforward to verify, and can be found in many references of Kalman filters \cite{MH12}.
In order to simplify the notations, we ignore the time indices below and let  $J(X)=(I +HXH^T)^{-1}$. Then picking any symmetric matrix $A$, the perturbation in direction $A$ is given by 
\[
D_A J(X):=\frac{d}{dt}J(X+At)\big|_{t=0}=-JHAH^TJ.
\]
Therefore
\[
D_A \Kalman=A-AH^TJHX-XH^TJHA+XH^TJHAH^TJHX=(I-H^TJHX)^TA(I-H^TJHX)
\]
The Hessian is 
\begin{align*}
D_A^2 \Kalman&=-2AH^TJHA+2AH^TJHAH^TJHX+2XH^TJHAH^TJHA\\
&\phantom{==}-2XH^TJHAH^TJHAH^TJHX\\
&=-2(AH^TJ^{1/2}-XH^THAH^TJ^{1/2})\cdot(AH^TJ^{1/2}-XH^THAH^TJ^{1/2})^T\preceq 0.
\end{align*}
Therefore, as long as $X, X+A\succeq 0$, then the convexity holds:
\[
\Kalman (X)+\Kalman(X+A)\preceq 2\Kalman(X+\tfrac12 A). 
\]
When we require $A$ to be PSD, $D_A\Kalman\succeq 0$ implies the monotonicity of $\Kalman$. 
\end{proof}

\begin{lem}
\label{lem:matrix}
Suppose that  $A, C, D$ are PSD matrices, $C$ is invertible,  while $A\preceq [BCB^T+D]^{-1}$, then
\[
B^TAB\preceq C^{-1},\quad A^{1/2}D A^{1/2}\preceq I_d. 
\]
\end{lem}
\begin{proof}
From the condition, we have $A^{1/2}[BCB^T+D] A^{1/2}\preceq I_d$. Therefore our second claim holds. Moreover, 
\[
(B^TAB)C(B^TAB)\preceq B^TA^{1/2} A^{1/2}[BCB^T+D] A^{1/2} A^{1/2}B\preceq B^TA B. 
\]
This leads to our first claim by the next lemma.
\end{proof}

\begin{lem}
\label{lem:trivial}
Let $A$ and $B$ be PSD matrices, if
\begin{itemize}
\item $A\succeq I_d$, then $ABA\succeq B$.
\item $A\preceq I_d$, then $ABA\preceq B$. And for any real symmetric matrix $C$, $CAC\preceq C^2$. 
 \end{itemize}
\end{lem}
\begin{proof}
If the null subspace of $B$ is $D$ and $\bfP$ is the projection onto the complementary subspace $D^\bot$, then it suffices to show that  $(\bfP A\bfP) (\bfP B\bfP) (\bfP A\bfP) \succeq \bfP B\bfP$. Therefore, without loss of generality, we can assume $B$ is invertible, so it suffices to show 
\[
(B^{-1/2} A B^{1/2})(B^{-1/2} A B^{1/2})^T\succeq I.
\]
But this is equivalent to checking the singular values of $B^{-1/2}A B^{1/2}$ are greater than $1$, which are the same as the  eigenvalues of $A$.

If $A$ and $C$ are invertible, then the second claim follows as the direct inverse of the first claim. Else, it suffice to show the claim on the subspace where $A$ and $C$ are invertible. 
\end{proof}
\begin{lem}
\label{lem:norm}
Let $A$ and $B$ be two PSD matrices, and $A$ is invertible, then
\[
\|A B\|=\|A^{1/2} B A^{1/2}\|=\inf\{\lambda: B\preceq \lambda A^{-1}\}. 
\]
\end{lem}
\begin{proof}
$\|A B\|=\|A^{1/2} B A^{1/2}\|$ comes as conjugacy preserves eigenvalues, and $\|A^{1/2} B A^{1/2}\|=\inf\{\lambda: B\preceq \lambda A^{-1}\}$ is obvious. 
\end{proof}

\begin{lem}
\label{lem:cond}
Let $A$ be a $d\times d$ PSD matrix, and $\Theta$ is a $p\times d$ matrix with rank $p$. Then the condition number of $\Theta A\Theta^T$ is dominated by the one of $A$.
\end{lem}
\begin{proof}
Suppose $v_M$ and $v_m$ are the p-dimensional eigenvectors of $\Theta A\Theta^T$ with the maximum  and minimum eigenvalues, $\lambda_M$ and $\lambda_m$. Then by
looking at $(\Theta^T v_M)^T A (\Theta^Tv_M)$ and $(\Theta^T v_m)^T A (\Theta^Tv_m)$ we know the maximum eigenvalue of $A$ is above $\lambda_M$, and the minimum eigenvalue of $A$ is below $\lambda_m$. 
\end{proof}

\bibliographystyle{unsrt}
\bibliography{EnKF}
\end{document}